\documentclass[12pt]{amsart}

\usepackage{amsmath,amsthm,amsfonts,amssymb,amscd}
\usepackage[all]{xy}
\usepackage{sseq}

\author{Thomas John Baird}
\title{Moduli spaces of vector bundles with fixed determinant over a real curve}

\theoremstyle{plain}
\newtheorem{thm}{Theorem}[section]
\newtheorem{prop}[thm]{Proposition}
\newtheorem{lem}[thm]{Lemma}
\newtheorem{cor}[thm]{Corollary}

\theoremstyle{definition}

\theoremstyle{remark}

\numberwithin{equation}{section}

\newcommand{\Z}{\mathbb{Z}}

\newcommand{\Q}{\mathbb{Q}}
\newcommand{\R}{\mathbb{R}}

\newcommand{\C}{\mathbb{C}}

\newcommand{\SL}{\mathcal{L}}

\newcommand{\ttau}{\tilde{\tau}}
\newcommand{\pC}{\mathcal{C}}

\renewcommand{\phi}{\varphi}

\newcommand{\tensor}{\otimes}

\DeclareMathSymbol{\sdp}{\mathbin}{AMSb}{"6F}

\newcommand{\G}{\mathcal{G}}

\newcommand{\E}{\mathcal{E}}
\newcommand{\F}{\mathbb{F}}

\newcommand{\ignore}[1]{}

\begin{document}
\baselineskip=16pt

\maketitle

\begin{abstract}
Let $(\Sigma,\tau)$ denote a Riemann surface of genus $g \geq 2$ equipped with an anti-holomorphic involution $\tau$. 
In this paper we study the topology of the moduli space $M(r,\xi)^\tau$ of stable Real vector bundles over $(\Sigma,\tau)$ of rank $r$ and fixed determinant $\xi$ of degree coprime to $r$. 

We prove that $M(r,\xi)^{\tau}$ is an orientable and monotone Lagrangian submanifold of the complex moduli space $M(r,\xi)$ so it determines an object in the appropriate Fukaya category. We derive recursive formulas for the $\Z_2$-Betti numbers of $M(r,\xi)^\tau$ and compute $\Z_p$-Betti numbers for odd $p$ through a range of degrees. We deduce that if $r$ is even and $ g >>0$, then $M(r,\xi)^{\tau}$ and $M(r,\xi')^{\tau}$ have non-isomorphic cohomology groups unless $\xi$ and $\xi'$ have equivalent Stieffel-Whitney classes modulo automorphisms of $(\Sigma,\tau)$. If $r$ is even, and $g>>0$ is even, we prove that the Betti numbers of $M(r,\xi)^{\tau}$ distinguish topological types of $(\Sigma, \tau; \xi)$. If $r=2$ and $g$ is odd, we compute all $\Z_p$-Betti numbers of $M(2,\xi)^\tau$. 
 
MR 32L05, 14P25.
\end{abstract}

\section{Introduction}\label{intro}

Let $\Sigma$ denote a Riemann surface of genus $g \geq 2$ and let $M(r,d)$ the moduli space of semi-stable holomorphic vector bundles over $\Sigma$ of rank $r$ and degree $d$. For simplicity, we assume throughout this introduction that $r$ and $d$ are coprime, which implies that $M(r,d)$ is a non-singular projective variety (we relax this condition in the rest of the paper). Given $\xi \in Pic_d(\Sigma) = M(1,d) $, denote by $M(r, \xi)$ the moduli space of rank $r$ bundles with fixed determinant $\xi$.  We may regard $M(r, \xi)$ as a fibre $det^{-1}(\xi)$ of the fibre bundle 
\begin{equation}\label{determap}
det: M(r,d) \rightarrow M(1,d)
\end{equation}
which sends the isomorphism class $[ \mathcal{E}]$ to $[\wedge^r \mathcal{E}]$.
A line bundle $\eta \in Jac(\Sigma) = M(1,0)$, determines an isomorphism $M(r,\xi) \cong M(r,\xi \otimes \eta^r)$. In particular, the subgroup $Jac[r] \leq Jac(\Sigma)$, of $r$-th roots of unity acts naturally on $M(r,\xi)$. Tensor product determines an isomorphism
$$  M(r,d) \cong  M(r,\xi) \times_{Jac[r]} M(1,0), $$
where the right side is a the mixed quotient with respect to tensor product actions on both factors.

In \cite{AB}, Atiyah and Bott calculated the cohomology groups of $M(r,d)$ and $M(r,\xi)$. In particular they proved that:
\begin{enumerate}
\item Both $H^*(M(r,d);\Z)$ and $H^*(M(r,\xi);\Z))$ are torsion free.
\item The action of $Jac[r]$ on $H^*(M(r,\xi);\Z))$ is trivial.
\item $H^*(M(r,d);\Z) \cong H^*(M(r,\xi);\Z)) \otimes H^*( M(1,d);\Z)$.
\end{enumerate}
Indeed, (3) follows from (1) and (2) by a simple argument. One of the goals of the present paper is to explore to what degree these properties hold when $M(r,d)$ is replaced by a moduli space of Real bundles over a real curve, as defined in \cite{BHH, Sch}.

A \emph{real curve} $(\Sigma, \tau)$ is a Riemann surface $\Sigma$ equipped with an antiholomorphic involution $\tau$. The fixed point set $\ \Sigma^\tau$ is a union of circles, called the \emph{real circles} of $(\Sigma,\tau)$. There is an induced antiholomorphic involution on $M(r,d)$ (which we also denote $\tau$), defined by
$$ \tau([E])  = [ \tau^*\overline{E}].$$
If $\xi \in M(1,d)$ is fixed by $\tau$, then $\tau$ restricts to an involution on $M(r, \xi)$, which we also denote $\tau$. The fixed point sets by $M(r,d)^{\tau}$ and $M(r,\xi)^{\tau}$ are half dimensional real submanifolds of $M(r,d)$ and $M(r,\xi)$ respectively. 

If $\Sigma^{\tau}$ has $a>0$ path components, then $M(r,d)^{\tau}$ has $2^{a-1}$ path components, parametrized by cohomology clases $w \in H^1(\Sigma^{\tau};\Z_2)$ which satisfy 
\begin{equation}\label{oddSW}
w(\Sigma^{\tau}) \equiv d~mod~2.
\end{equation}
Denote by $M(r,d)^{\tau}_w$ the path component corresponding to $w$. The holomorphic bundles $\mathcal{E} \in M(r,d)^{\tau}$ are precisely those that admit an anti-holomorphic bundle automorphism $\ttau$ lifting $\tau$ and we call such $\mathcal{E}$ \emph{(holomorphic) Real  vector bundles}. The invariant $w$ is simply the first Stiefel-Whitney class of the real locus $E^{\ttau} \rightarrow \Sigma^{\tau}$.  We call a real circle $S \subseteq \Sigma^{\tau}$ \emph{odd} (resp. \emph{even}) with respect to $E \in M(r,d)^{\tau}_w$  if $w(S) =1$ (resp. $w(S) =0$). 

If $\Sigma^{\tau}$ is empty, $M(r,d)^{\tau}_0$ exists as before whenever $d$ is even, but there may also be a path component $M(r,d)_{\mathbb{H}}^\tau$ corresponding to what are called Quaternionic vector bundles (see \cite{BHH} for a details). However tensoring with a Quaternionic line bundle of degree $d'$ determines isomorphism $M(r,d)_{\mathbb{H}}^{\tau} \cong M(r,d +rd')_{0}^{\tau}$, so these Quaternionic vector bundles can safely be neglected for our purposes.

As a Lie group, $M(1,0)_0^{\tau} \cong U(1)^g$. Let $\Gamma_g \cong (\Z/r)^g$ be the $r$-torsion subgroup of $M(1,0)_0^\tau$. Note that $\Gamma_g$ is a subgroup of $Jac[r] \cap M(1,0)^{\tau}$.  We have an isomorphism 
$$ M(r,d)^{\tau}_w \cong M(r,\xi)^{\tau} \times_{\Gamma_g} M(1,0)_0^{\tau}$$
where $\xi \in M(1,d)_w^{\tau}$ and the right side is a the mixed quotient with respect to tensor product actions on both factors (see \cite{B3}, \S 6).
Our first main theorem is a version of (2) and (3) for mod 2 coefficients.

\begin{thm}\label{thm1}
The action of $\Gamma_g$ on $H^*(M(r,\xi);\Z_2)$ is trivial and we have an isomorphism
$$  H^*(M(r,d)^{\tau}_w;\Z_2 ) \cong H^*(M(r,\xi)^{\tau};\Z_2) \otimes H^*( M(1,0)_0^{\tau};\Z_2).$$
\end{thm}

Since recursive formulas for the mod 2 Betti numbers of $M(r,\xi)^{\tau}$ were computed in \cite{B, LS}, Theorem \ref{thm1} yields formulas for the Betti numbers of $M(r,\xi)^{\tau}$. We present explicit formulas for $r =2$ and $r= 3$ in \S \ref{Betti numbers and fundamental groups} for convenience.

One of the peculiarities of the mod 2 Betti numbers formulas \cite{B, LS} is that they depend only on the real curve $(\Sigma, \tau)$ and not on the Stiefel-Whitney class $w$. When $r$ is odd this can be explained by the existence of homeomorphisms $M(r,\xi)^{\tau} \cong M(r,\xi')^{\tau}$ for any pair of Real line bundles $\xi$ and $\xi'$, determined by tensoring with third Real line bundle $\eta$ such that $\xi \cong \xi' \otimes \eta^r$. However, such homeomorphisms generally do not exist when $r$ is even.

\begin{thm}\label{thm6}
For $i \in \{1,2\}$ let $(\Sigma_i,\tau_i)$ be real curves of genus $g_i$ and  $a_i$ real circles, and let $\xi_i$ be Real line bundles over $(\Sigma_i, \tau_i)$ with $c_i$ many even circles. Then we have an isomorphism of graded groups 
\begin{equation}\label{whenisom}
H^*(M(r,\xi_1)^{\tau_1};\Z) \cong H^*(M(r,\xi_2)^{\tau_2};\Z)
\end{equation} 
only if $g_1=g_2$, $a_1=a_2$. 

Suppose additionally that $r$ is even and that either $r =2$ and $g_1 \geq 5$ or $r \geq 4$ and $g_1 \geq 3$. Then (\ref{whenisom}) holds only if $c_1=c_2$. 

Suppose in further addition that $g$ is even and $g \geq 6$ . Then (\ref{whenisom}) holds only if there exists a homeomorphism $f: \Sigma_1 \rightarrow \Sigma_2$ such that $f \circ \tau_1 = \tau_2 \circ f$. 
\end{thm}

Theorem \ref{thm6} is proven by computing the odd characteristic Betti numbers in all degrees less than $g(r-1)-1$. For rank $2$ bundles and odd genus, we can do better and compute the entire Poincare polynomial. 

\begin{thm}\label{thm4}
Let $(\Sigma,\tau)$ be a real curve of odd genus and let $\xi$ be a Real line bundle of odd degree for which $\xi$ has $c$ even circles. For any field $\F$ of characteristic $\neq 2$ we have
\begin{equation}\label{oddgbetti}
 P_t( M(2,\xi)^{\tau};\F) =   \frac{1}{2}(1+t^{3})^{g-c-1}  ((1+t)^{c} (1+ t^2)^{c} + (1- t)^{c} (1- t^2)^{c}) .
\end{equation}
\end{thm}
Note that the Betti number formula (\ref{oddgbetti}) fails to fully distinguish between topological types of real curves, in contrast to what happens when $g$ is even. In particular, if for $i \in \{ 1,2\}$, $(\Sigma_i, \tau_i,\xi_i)$ are real curves with the same odd genus $g$, the same number of real circles $a$, equipped with Real line bundles with the same number of even circles $c$, but  $ \Sigma_1 \setminus \Sigma_1^{\tau_1}$ is connected and $ \Sigma_2 \setminus \Sigma_2^{\tau_2}$ is disconnected, then the corresponding moduli spaces $M(2,\xi_1)^{\tau_1}$ and $M(2;\xi_2)^{\tau_2}$ have identical Betti numbers in all characteristics. This is a peculiar fact for which I have no moral explanation.

In \cite{B3} the Poincar\'e polynomial of the invariant subring $H^*(M(2,\xi)^{\tau};\Q)^{\Gamma_2}$ was shown to equal $(1+t^3)^{g-1}$, when $c \neq 0$. We deduce that the real analogue of (2) is false.

\begin{cor}\label{thm2}
Let $\F$ be a field of characteristic $\neq 2$. Then the action of $\Gamma_g$ on $H^*( M(2,\xi)^{\tau};\F)$ is non-trivial in general.
\end{cor}

We also calculate the fundamental group of  $M(r,\xi)^{\tau} $ except when $r=2$ and $g=2$. A consequence is that the real analogue of (1) is false.

\begin{thm}\label{thm3}
Let $(\Sigma, \tau)$ denote a real curve of genus $g \geq 2$ with $a$ real circles, let  $\xi \in M(1,d)^{\tau}$  with $b$ odd circles, and let $r\geq 2$.  We have an isomorphism 
$$  \pi_1(M(r, \xi)^{\tau})  \cong    \begin{cases}  \Z/2 \ltimes (\Z/2)^a & \text{if $r\geq 3$}\\ \Z/2 \ltimes ( (\Z/2)^b \times (\Z)^{a-b}) & \text{if $r = 2$}.\\ \end{cases} $$
where $\Z/2$ acts diagonally: trivially on the $\Z/2$ factors and by $-1$ on the $\Z$ factors.

Consequently, $H_1(M(r,\xi)^{\tau}; \Z) \cong (\Z/2)^a$ unless $g=2$ and $r=2$. 
\end{thm}
Note that the conclusion of Theorem \ref{thm3} does not extend to the case $g=2$ and $r=2$. This case is completely worked out in \cite{BCH} where in some examples $H_*(M(r,\xi)^{\tau}; \Z)$ is torsion-free.

The strategy for proving all of the above results is to study the real Harder-Narasimhan stratification introduced in \cite{LS,B}, which is a real analogue of the complex Harder-Narasimhan stratification studied by Atiyah and Bott \cite{AB}. This stratification relates the topology of $M(r,d)^{\tau}$ with the topology of a group of real gauge transformations, $\G_\R$, and was used in \cite{LS,B} to compute $\Z_2$-Betti numbers of $M(r,d)^{\tau}$ and in \cite{B3} to compute the $\Q$-Betti numbers of $M(2,d)^\tau$. In similar fashion, we relate the topology of $M(r,\xi)^{\tau}$ with $C \G_R$, the group of real gauge transformations with constant determinant.

One motivation for studying these fixed determinant moduli spaces is that they form a rich and geometrically interesting class of real Lagrangian submanifolds of $M(r,\xi)$ endowed with the standard Atiyah-Bott symplectic form. In \S \ref{Orientability and Monotonicity}, we prove a couple results that ensure these $M(r,\xi)^{\tau}$ have well-defined Lagrangian Floer cohomology over $\Z_2$-coefficients and thus determine objects in the appropriate Fukaya category \cite{BC,FOOO}.

\begin{thm}\label{orientmonotone}
$M(r,\xi)$ is orientable and monotone with minimal Maslov index a positive multiple of two.
\end{thm}

We have not been able to prove that $M(r,\xi)^\tau$ is relatively spin in general, so the Fukaya category is defined only with $\Z_2$ coefficients. However in \cite{BCH} we prove that when $r=2$, $M(2,\xi)^\tau$ is relatively spin in $M(2,\xi)$ so it determines an object in a $\Z$-Fukaya category.

We summarize the contents of the paper. In \S \ref{Strategy for proving Theorems} we outline the basic strategy relating $M(r,\xi)$ to the real Harder-Narsimhan filtration and the classifying space of the real gauge group $C \G_\R$. The technical heart of the paper is \S \ref{Topology of BCG} where we compute the Betti numbers of $BC \G_\R$ using an Eilenberg-Moore spectral sequence and also compute the fundamental group, yielding proofs of Theorems \ref{thm6} and \ref{thm3}. In \S \ref{Equivariant perfection} we prove that the real Harder-Narasimhan filtration is $C\G_\R$-equivariantly perfect with respect to mod 2 coefficients, completing the proof of Theorem \ref{thm1}. In \S \ref{Proof of thm4} we prove Theorem \ref{thm4} by showing that the Thom spaces of the unstable strata are $\F$-acyclic. In the remaining sections we illustrate our results with some examples and prove Theorem \ref{orientmonotone}.

\textbf{Notational conventions:}  
If $G$ is a topological group acting on a topological space $X$ we denote $X_{hG} = EG \times_G X$ the homotopy quotient.  Denote the Poincar\'e series $P_t(X;\F) = \sum_{i=0}^{\infty} \dim(H^i(X;\F))t^i$.

\section{Basic strategy}\label{Strategy for proving Theorems}

We recall the construction of $M(r,d)_w^{\tau}$ from \cite{ BHH, Sch}. We no longer require $gcd(r,d) =1$.

Fix a real curve $(\Sigma, \tau)$. Topologically, real curves are classified (see \cite{W}) by invariants $(g,a,\epsilon)$ where $g \geq 0$ is the genus of $\Sigma$, $a = \pi_0(\Sigma^\tau)$ is the number of real circles, and $\epsilon =1$ if $\Sigma \setminus \Sigma^{\tau}$ is connected and $\epsilon =0$ if $\Sigma \setminus \Sigma^{\tau}$ is disconnected. A real curve with invariants $(g,a,\epsilon)$ exists if and only if $1-\epsilon \leq a \leq g+1-\epsilon$ and $ g+1 \equiv a~mod~2$ if $\epsilon= 0$. 

Fix a smooth complex vector bundle $\pi: E\rightarrow \Sigma$ of rank $r$ and degree $d$ endowed with an anti-linear bundle $\ttau: E \rightarrow E$ such that $\pi \circ \ttau = \tau \circ \pi$. We call $(E,\ttau)$ a $C^{\infty}$-Real vector bundle over $(\Sigma, \tau)$. The fixed point set $E^{\ttau} \rightarrow \Sigma^{\tau}$ is a $\R^r$-bundle and we require that $w = w_1(E^{\ttau})$. Topologically, $(E,\ttau)$ is classified (\cite{BHH}) by $d$ and $w$, subject to the condition that 
\begin{equation}\label{d=w}
 d \equiv w(\Sigma^{\tau})~mod~2.
\end{equation} 
and $w(\Sigma^{\tau})$ is equal the number of odd circles for $w$ (see \cite{BHH}).

Denote by $\pC =\pC(E) = \pC(r,d)$ (the Sobolev completion of) the space holomorphic structures on $E$, represented by $L^2_s$-Cauchy-Riemann operators $\bar{\partial}$ on $E$ for some fixed $s >1$. Denote by $\pC^{\ttau}$ the subspace of holomorphic structures that commute with $\ttau$, which we call Real holomorphic structures. As topological spaces both $\pC$ and $\pC^{\ttau}$ are contractible Banach manifolds. $\pC^{\ttau}$ is acted upon by the \emph{real gauge group} $$\G_{\R}= \G(E)_\R =  \G^{\ttau},$$ consisting of $L_{s+1}^2$-gauge transformations that commute with $\ttau$. 

In case $E= L$ is a line bundle, there is a natural isomorphism $\G(L) \cong Maps(\Sigma, \C^*)$, so the isomorphism type of $\G(L)$ is independent of $L$. If $(L,\ttau)$ is a Real line bundle over $(\Sigma, \tau)$, then $\G(L)^{\ttau}$ is identified with maps that are equivariant with respect to involutions on $\Sigma$ and $\C^*$, so $\G(L)^{\ttau} = Maps_{\Z/2}(\Sigma, \C)$ is also independent of $(L,\ttau)$.  We write $$ \G(1)_{\R} = \G(L)^{\ttau}$$ to make this independence explicit.

$\pC^{\ttau}$ admits a $\G_\R^{\ttau}$-equivariant stratification $\bigcup_\mu \pC^{\ttau}_{\mu}$  according to real Harder-Narasimhan type (\cite{B} \S 2). This stratification is equivariantly perfect with respect to the $\G_{\R}$-action and $\Z_2$-coefficients. This means that
$$ P_t(\pC^{\ttau}_{h\G_\R};\Z_2) = \sum_{\mu} t^{d_\mu} P_t((\pC^{\ttau}_\mu)_{h\G_\R};\Z_2)  $$
where $d_\mu$ is the codimension of $\pC_\mu$ in $\pC$. Since the central subgroup of scalars $ \R^* \leq \G_{\R}$ acts trivially it is sometimes preferable to work with the quotient group $\overline{\G}_{\R}  = \G_{\R}/\R^*$ which acts effectively. The stratum $\pC^{\ttau}_{ss}$ consisting of those Real holomorphic structures that are geometrically semistable is dense and open. The $\G_{\R}$-action restricts to  $\pC^{\ttau}_{ss}$ with orbit space $$\pC^{\ttau}_{ss}/\G_{\R} = \pC^{\ttau}_{ss}/\overline{\G}_{\R} = M(r,d)^\tau_w.$$

If $gcd(r,d)=1$, then $\overline{\G}_\R$ acts freely on $\pC^{\ttau}_{ss}$ and the quotient exact sequence $1 \rightarrow \R^* \rightarrow \G_{\R} \rightarrow \overline{\G}_{\R} \rightarrow 1$  splits (\cite{B}, Lemma 7.1), so we have a non-canonical isomorphism 
\begin{equation}\label{noncansplit}
\G_{\R} \cong \overline{\G}_{\R} \times \R^*.
\end{equation}

Consider now the subgroup $C\G_{\R} \leq \G_{\R}$ of real gauge transformations with constant determinant. These are the gauge transformations of $E$ that act as a constant scalar multiplication on the determinant line bundle $\Lambda^rE$, so $C\G_{\R} $ fits into a short exact sequence
\begin{equation}\label{defineCG}
1 \rightarrow  C\G_{\R} \rightarrow \G_{\R} \rightarrow \overline{\G(1)}_{\R} \rightarrow 1,
\end{equation}
where surjectivity of $\G_{\R} \rightarrow \overline{\G(1)}_{\R} = \overline{\G( \Lambda^rE)}_{\R}$ follows by considering a Whitney sum decomposition of $E$ into Real line bundles (see (\ref{decompofrealbundles})). We will later need the following.

\begin{lem}\label{homtype1}
The group of path components $ \pi_0( \overline{ \G(1)}_\R)$ is isomorphic to $\Z^g$ and the identity component of $ \overline{ \G(1)}_\R$ is contractible. Therefore $B \overline{ \G(1)}_\R \cong (S^1)^g$.
\end{lem}

\begin{proof}
Since $\overline{\G(1)}_{\R}$ acts freely on the contractible space $C(L)^{\ttau} = C(L)^{\ttau}_{ss}$ it follows that  $$B \overline{\G(1)}_{\R} = C(L)^{\ttau}/ \overline{\G(1)}_{\R} = M(1,0)^\tau_w. $$ Since $M(1,0)_w^{\tau}$ is homeomorphic to $(S^1)^g$ it follows that $B \overline{\G(1)}_{\R}$ is a $K(\Z^g, 1)$ and thus that the quotient map $ \overline{\G(1)}_{\R} \rightarrow \pi_0( \overline{\G(1)}_{\R} ) \cong \Z^g$ is a homotopy equivalence.
\end{proof}

The scalar transformations are contained in  $C \G_{\R}$ so defining $ \overline{C\G}_{\R} = C \G_{\R}/\R^*$, gives rise to a short exact sequence
$$1 \rightarrow  \overline{C\G}_{\R} \rightarrow \overline{\G}_{\R} \rightarrow \overline{\G(1)}_{\R} \rightarrow 1.$$

 If $gcd(r,d) =1$ we have with a non-canonical isomorphism
 \begin{equation} \label{noncansplit}
C \G_{\R} \cong \overline{C\G}_{\R} \times \R^*.
\end{equation}

\begin{lem}\label{lemstuff}
Let $(E,\ttau)$ be a  Real $C^{\infty}$-vector bundle  of  rank $r$ and degree $d$ with $w := w_1(E^{\ttau})$, and let $\xi \in M(1,d)_{w}^{\tau}$. Then there is a homotopy equivalence $ M(r, \xi)  \cong \pC^{\ttau}_{ss}/C\G_{\R} \cong  \pC^{\ttau}_{ss}/\overline{C\G}_{\R}$. 
\end{lem}

\begin{proof}
Consider the determinant map $\det: \pC^{\ttau} \rightarrow \pC(\Lambda^rE)^{\ttau}$. This is equivariant with respect to $\overline{\G}_{\R}$ and $\overline{C\G}_{\R}$ is the stabilizer of every point in $\pC(\Lambda^rE)^{\ttau}$. Consequently, we can identify $(\pC^{\ttau}_{ss})/\overline{C\G}_{\R} $ as the pull-back of the diagram

$$ \xymatrix{ &    \pC^{\ttau}_{ss}/\overline{\G}_{\R} = M(r,d)_{w}^\tau\ar[d] \\
\pC(\Lambda^rE)^{\ttau}\ar[r]  &  \pC(\Lambda^rE)^{\ttau}/\overline{\G}_{\R}  = M(1,d)_{w}^\tau}.$$
Since both morphisms in the diagram are fibre bundles, the pull-back is homotopy equivalent to the homotopy pull-back.  Since $\pC(\Lambda^rE)^{\ttau}$ is contractible, we conclude that $\pC^{\ttau}_{ss}/\overline{C\G}_{\R}$ is homotopy equivalent to the fibre of the determinant map $M(r,d)_{w}^\tau \rightarrow M(1,d)_{w}^\tau$.

\end{proof}

\begin{cor}\label{homotequivMC}
With notation as in Lemma \ref{lemstuff}, if $gcd(r,d)=1$ then we have a homotopy equivalences  $M(r,\xi) \cong (C_{ss})_{h \overline{C\G}_\R}$ and  $M(r,\xi) \times B \R^* \cong (C_{ss})_{h C\G_\R}$.
\end{cor}

\begin{proof}
If $gcd(r,d)=1$ then $C\G_\R \cong \overline{C\G}_\R \times \R^*$ where $\R^*$ acts trivially and $\overline{C\G}_\R$ acts freely. The result now follows from Lemma \ref{lemstuff}.
\end{proof}

The strategy for proving Theorem \ref{thm1} is as follows. We have diagram of homotopy quotients

$$ \xymatrix{  (\pC^{\ttau}_{ss})_{h \overline{C\G}_\R} \ar[r] \ar[d] &   (\pC^{\ttau})_{h \overline{C\G}_\R}  \ar[d]\\
  (\pC^{\ttau}_{ss})_{h \overline{\G}_\R} \ar[r] &  (\pC^{\ttau})_{h \overline{\G}_\R} .} $$ 
where arrows are induced by inclusions $\pC^{\ttau}_{ss} \hookrightarrow  \pC^{\ttau}$ and $\overline{C\G}_\R \hookrightarrow \overline{\G}_\R$. If $gcd(r,d) =1$, then this diagram is equivalent up to homotopy to
\begin{equation}\label{LHdiagram}
 \xymatrix{ M(r,\xi)^\tau \ar[r] \ar[d]^{i}  &   B\overline{C\G}_\R \ar[d]\\
 M(r,d)_w^{\tau} \ar[r] & B\overline{\G}_\R .} 
 \end{equation}
Here $i$ can be identified with the fibre inclusion (\ref{determap}). We will show that all of the maps in (\ref{LHdiagram}) induce $\Z_2$-cohomology surjections. Theorem \ref{thm1} then follows by the Leray-Hirsch Theorem.

To prove our results on odd characteristic cohomology, we use the following.

\begin{cor}\label{toprovetheorem3}
If $gcd(r,d) =1$ then the map (\ref{LHdiagram}) $M(r,\xi) \rightarrow B\overline{C\G}_\R$ induces a surjection on $\pi_k$ for $k\leq g(r-1)-1$ and an isomorphism for $k \leq g(r-1)-2$. Consequently $$H_k(M(r,\xi);\Z) \rightarrow  H_k(BC\G_\R; \Z)$$and 
$$ H_k(BC\G_\R; \F) \rightarrow H_k(M(r,\xi);\F) $$ 
are isomorphisms for all $k \leq g(r-1)-2$ and coefficient fields $\F$.
\end{cor}

\begin{proof}
The codimension of all unstable strata is always greater than or equal to $g(r-1)$ (an easy exercise given the codimension formula (2.4) in \cite{B}). Therefore the induced map
$$  (\mathcal{C}_{ss}^{\ttau})_{h \overline{C\G}_\R} \rightarrow  (\pC^{\ttau})_{h \overline{C\G}_\R} \cong  B\overline{C\G}_\R$$ must be be a surjection on $\pi_k$ for $k = g(r-1)-1$ and an isomorphism for $k \leq g(r-1)-2$. The result follows from Corollary \ref{homotequivMC}, the Hurewicz Theorem, and the Universal Coefficient Theorem.
\end{proof}

\section{Topology of $BC\G_\R$}\label{Topology of BCG}

Let $\G_\R = \G(E)_\R$. In this section, we compute the Betti numbers of $BC\G_\R$ in all characteristics $p$. We begin with material that is independent of $p$ and then treat $p=2$ and $p>2$ in turn. Much of this section is adapted from calculations in \cite{B} and \cite{B3}, to which we sometimes refer for details.

Recall that $C\G$ is the group of gauge transformations of $E$ with constant determinant.  This fits into a short exact sequence
$$1 \rightarrow S\G \rightarrow C\G \rightarrow \C^* \rightarrow 1 $$
where $S\G$ is the group of gauge transformations with determinant 1.  Likewise we have a short exact sequence
\begin{equation}\label{SESCSR}
 1 \rightarrow S\G_\R \rightarrow C\G_\R \rightarrow \R^* \rightarrow 1
 \end{equation}
where $S\G_\R$ and $C\G_\R$ are the subgroups of $S\G$ and $C\G$ respectively that commute with $\ttau$. 

From the classification of $C^{\infty}$-Real vector bundles over a real curve in \cite{BHH}, it is always possible to decompose into Real subbundles
\begin{equation}\label{decompofrealbundles}
E = E' \oplus L
\end{equation} where $L$ is a Real line bundle. Define a splitting of (\ref{SESCSR}) by lifting $\lambda \in \R^*$ to the real gauge transformation which is trivial on $E'$ and scalar multiplication by $\lambda$ on $L$. This implies that $C \G_{\R}$ is isomorphic to a semi-direct product $ \R^* \ltimes S\G_\R$.

Suppose now that $E$ is endowed with a $\ttau$-equivariant Hermitian metric and let $S\mathcal{U}_\R \leq S\G_\R$ be the subgroup of elements that act unitarily. This inclusion is a homotopy equivalence, because $S\G_\R/S\mathcal{U}_\R$ can be identified with the convex space of $\tau$-compatible Hermitian metrics, so it induces a homotopy equivalence $$B S\mathcal{U}_{\R} \cong B S\G_\R.$$
For technical reasons to do with compactness, it is preferable to work with $S\mathcal{U}_{\R}$. 

Let $X$ denote a compact orientable 2-manifold of genus $\hat{g}$ with $n$ boundary components, where $2\hat{g} +n-1 = g$.  Consider the pull-back diagram of groups
\begin{equation}\label{gcdiag}
\xymatrix{ S\mathcal{U}(X , r; \tau_1,...,\tau_n) \ar[r] \ar[d] & Maps(X , SU_r) \ar[d]^{\pi}\\
\prod_{i=1}^n LSU_r^{\tau_i} \ar[r]^{\iota} &  \prod_{i=1}^n LSU_r  }
\end{equation}
where $Maps(X , SU_r)$ is the space of continuous maps from $X$ to $SU(r)$ with pointwise multiplication, $LSU_r$ is the space of continuous maps from $S^1 = \R/2 \pi \Z$ into $SU_r$, $\pi$ is restriction onto the boundary circles numbered $1$ to $n$, and $\iota$ is the product of inclusions of some choice of real loop groups $LSU_r^{\tau_i} \leq LSU_r$ that will be introduced shortly. Applying the classifying space functor yields a homotopy pull-back diagram 
\begin{equation}\label{hpd}
\xymatrix{ BS\mathcal{U}(X , r; \tau_1,...,\tau_n) \ar[r] \ar[d] & B Maps(X , SU_r) \ar[d]^{B\pi}\\
\prod_{i=1}^n BLSU_r^{\tau_i} \ar[r]^{B\iota} &  \prod_{i=1}^n BLSU_r  .}
\end{equation}

We must now describe the Real loop groups $LSU_r^{\tau_i}$. These are subgoups of $LSU_r$ and come in three types:

\begin{itemize}
\item[($\alpha$)] $LSU_r^{{\ttau}_\alpha} = LSO_r$ sitting inside $LSU_r$ in the standard way,
\item[($\beta$)] $LSU_r^{{\ttau}_\beta} = L_{-1}SO_r$ is the group of locally orientation preserving gauge transformations of a M\"obius bundle $\R^r \rightarrow M \rightarrow S^1$. It injects into $LSU_r$ via an isomorphism $M \otimes_\R \C \cong \C^r \times S^1$. 
\item[($\gamma$)] $LSU_r^{{\ttau}_\gamma} = \{ g: S^1 \rightarrow U_r |  g(\theta) = \overline{g(\theta + \pi)}\}$ where the bar means entry-wise complex conjugation. 
\end{itemize}

\begin{lem}\label{gaugegroupconstr}
For some choice of $\tau_1,...,\tau_n \in \{ \tau_\alpha, \tau_\beta,\tau_\gamma\}$, there is an isomorphism $S\mathcal{U}_\R \cong S\mathcal{U}(X , r; \tau_1,...,\tau_n) $ that induces a homotopy equivalence 
$$ BS\G_\R \cong BS\mathcal{U}_\R \cong BS\mathcal{U}(X , r; \tau_1,...,\tau_n) .$$ 
There is one real loop group of type ($\alpha$) for each real component of $\Sigma^{\tau}$ over which $E^{{\ttau}}$ is trivial,  one of type ($\beta$) for each real component for which $E^{{\ttau}}$ is nonorientable, and a positive number of type ($\gamma$) if and only if $\Sigma \setminus \Sigma^{\tau}$ is connected.
\end{lem}

\begin{proof}
This proven the same way as (\cite{B} Proposition 6.2) except that $U(r)$ is replaced by $SU(r)$. 
\end{proof}

Our plan is to compute $H^*(BS\G(X , r; \tau_1,...,\tau_n))$ using the Eilenberg-Moore spectral sequence (EMSS) associated to (\ref{hpd}). 

\begin{lem}\label{blsu}
Over any coefficient field we have an isomorphism $ H^*(BLSU_r)  \cong \wedge(\bar{c}_2,...,\bar{c}_r) \otimes S(c_2,...,c_r) $, where the generators have degrees $|\bar{c}_{k}| = 2k-1$ and $|c_{k}| = |c_k| = 2k$.
\end{lem}

\begin{proof}
Restricting to the basepoint determines a fibration sequence
\begin{equation}\label{sfibseq1}
 SU_r \rightarrow B LSU_r  \rightarrow BSU_r,
\end{equation}
where we have identified $B\Omega SU_r$ with $SU_r$.
The inclusion $LSU_r \hookrightarrow LU_r$ induces a morphism of fibration sequences of (\ref{sfibseq1}) into 
\begin{equation}\label{fibwithu1}
  U_r \rightarrow B L U_r \rightarrow BU_r.
  \end{equation}
It was proven  in \cite{B} Proposition  4.3 (stated for $\Z_2$ coefficients, but the proof is valid in any characteristic) that the Serre spectral sequence of (\ref{fibwithu}) collapses yielding a ring isomorphism
$$ H^*(LU_r) \cong  H^*(U_r) \otimes  H^*(BU_r).  $$ 
Because $SU(r) \subseteq U(r)$ determines a surjection on cohomology, Leray-Hirsch yields a ring isomorphism
$$ H^*( B LSU_r)  \cong H^*(SU_r) \otimes H^*(BSU_r). $$
\end{proof}

For the rest of this section we use index sets, $i \in \{1,...,n\}$, $i' \in \{2,...,n\}$, $k \in \{2,...,r \}$.  We use the notational convention that the appearance of one of these subscripts means to include the full range of that index set. For example $\wedge(\bar{c}_2,...,\bar{c}_r) \otimes S(c_2,...,c_r)  = \wedge(\bar{c}_k) \otimes S(c_k)$.  

\begin{lem}\label{aolercubalcroe}
Over any field $\F$, we have an isomorphisms
\begin{equation}\label{kunn}
H^*(\prod_{i=1}^n BLSU_r) \cong \bigotimes_{i=1}^n H^*(BLSU(r)) \cong \wedge( \bar{c}_{i,k}) \otimes S(c_{i,k})
\end{equation}
and

\begin{equation}\label{mapspace}
H^*(B Maps(X , SU_r)) \cong \frac{\wedge( \bar{c}_{i,k}) }{(c_{1,k} +....+c_{n,k})}\otimes S(c_{k}) \otimes A.
 \end{equation}
where the generators have degrees $|\bar{c}_{i,k}| = 2k-1$ and $|c_{i,k}| = |c_k| = 2k$ and $A$ is an exterior algebra with Poincar\'e series $$P_t(A) = \prod_{k=2}^r (1+t^{2k-1})^{2 \hat{g}}.$$ 
In terms of these generators, the map $$B\pi^*:  H^*(\prod_{i=1}^n BLSU_r) \rightarrow H^*(B Maps(X , SU_r)) $$ is determined by $B\pi^*(\bar{c}_{i,k}) = \bar{c}_{i,k}$, and $B\pi^*(c_{i,k}) = c_k$.
\end{lem}

\begin{proof}
Equation (\ref{kunn}) follows from Lemma \ref{blsu} by the Kunneth Theorem.

To prove (\ref{mapspace}), first observe the homotopy equivalence
$$ X \sim \vee_{g} S^1 $$
between $X$ and a wedge of $g = 2\hat{g}+n-1$ circles. Thus
$$ BMaps(X,SU_r) \cong BMaps( \vee_{g} S^1 ,SU_r).$$
Restricting to the basepoint determines a fibration sequence
\begin{equation}\label{sfibseq}
 SU(r)^{g} \rightarrow B Maps(X , SU_r)  \rightarrow BSU(r).
\end{equation}
The inclusion $Maps(X, SU(r)) \hookrightarrow Maps(X,U(r))$ induces a morphism of fibration sequences of (\ref{sfibseq}) into 
\begin{equation}\label{fibwithu}
  U(r)^{g} \rightarrow B Maps(X , U_r)  \rightarrow BU(r).
  \end{equation}
It was proven (stated for $\Z_2$ coefficients, but the proof is valid in any characteristic) in \cite{B} Lemma  4.4 that the Serre spectral sequence of (\ref{fibwithu}) collapses yielding a ring isomorphism
$$ H^*(B Maps(X,U_r)) \cong  H^*( U(r)^{g}) \otimes  H^*(BU(r)).  $$ 
Because $SU(r) \leq U(r)$ determines a surjection on cohomology, Leray-Hirsch yields a ring isomorphism
$$ H^*( B Maps(X , SU_r))  \cong H^*( SU(r)^{g} ) \otimes H^*(BSU(r)). $$
Under the homotopy equivalence between $X$ and a wedge of circles, $(n-1)$ of the boundary circles of $X$ are sent to circles in the wedge, while the sum of the boundary circles is a boundary. The induced map on cohomology follows.
\end{proof}

The Koszul-Tate complex for the homomorphism $B\pi^*$ is identified with the bigraded complex $(K^{*,*},\delta)$ where
\begin{equation}\label{ktc}
K^{*,*} := \Gamma(z_k) \otimes \wedge(x_{i',k}) \otimes \wedge(\bar{c}_{i,k}) \otimes S(c_{i,k}) \otimes A
\end{equation}
with bidegrees and differentials

\begin{tabular}{|c|c|c|}
	\hline
generator   & bi-degree  & $\delta$-derivative \\
\hline
$\bar{c}_{i,k}$ & $(0,2k-1)$ & $0$ \\
$c_{i,k}$ & $(0,2k)$ & $0$\\
$x_{i',k}$ & $(-1,2k)$ & $c_{i',k}-c_{1,k}$ \\
$z_k$ & $(-1,2k-1)$ & $\bar{c}_{1,k}+...+\bar{c}_{n,k} $\\
\hline	
\end{tabular}

Note in particular that $K^{*,*}$ is a free extension over 
$$R^* := H^*( \prod_{i=1}^n BLSU_r ) \cong \wedge(\bar{c}_{i,k}) \otimes S(c_{i,k})$$ and the cohomology $H(K^{*,*},\delta)$ is isomorphic to $H^*(BMaps(X,SU_r))$ as a graded $R$-module, where we understand elements in $H^d(BMap(X,SU_r))$ to have bi-degree $(0,d)$. By homotopy pullback (\ref{hpd}) gives rise to a Eilenberg-Moore spectral sequence (EMSS), for which $EM_2^{*,*}$ is isomorphic as a bi-graded algebra to the homology of the complex 

\begin{equation}\label{ComplexEMSS}
 ( K^{*,*} \otimes_{R^*} H^{*}(\prod_{i=1}^n BLSU_r^{\tau_i}),~ \delta \otimes_{R^*} 1).
 \end{equation}

\subsection{Characteristic $2$}

\subsubsection{Cohomology of the loop groups over $\Z_2$.}

It follows from surjectivity into the non-fixed determinant case that the loop groops
$$ B\Omega SO_r  \rightarrow BL_{\sigma} SO_r \rightarrow BSO_r. $$
and
$$ B\Omega SU_r \rightarrow BLSU_r^{\tau} \rightarrow BSU_r $$
have Serre Spectral sequences that collapse.  Consequently,

\begin{prop}
We have isomorphisms
$$ H^*(BL_{\sigma} SO_r; \Z_2) \cong  H^*(SO_r;\Z_2) \otimes H^*(BSO_r;\Z_2)  $$
as modules over $H^*(BSO_r;\Z_2)$ and
$$ H^*(BLSU_r^{\tau};\Z_2) \cong  H^*(SU_r;\Z_2) \otimes H^*(BSU_r;\Z_2)  $$
as modules over $H^*(BSU_r;\Z_2)$.
\end{prop}

\begin{cor}
$H^*(BL_{\sigma} SO_r;\Z_2)$ is a free module over $H^*(BSU_r;\Z_2)$ with Poincar\'e polynomial
$$ P_t(BSU_r) \prod_{k=2}^{r} (1+t^{k-1})(1+t^{k}) , $$
and $H^*(BLSU_r^{\tau};\Z_2)$ is a free module over $H^*(BSU_r;\Z_2)$ with Poincar\'e polynomial
$$  P_t(BSU_r) \prod_{k=2}^{r} (1+t^{2i-1}) .$$
\end{cor}

\subsubsection{ Cohomology of $BC\G_\R$ over $\Z_2$}

\begin{thm}\label{surjf}
The inclusion $C\G_\R \leq \G_\R$ induces a surjection in cohomology $$ H^*(B\G_\R; \Z_2) \rightarrow  H^*(BC\G_\R; \Z_2).$$
\end{thm}

The short exact sequence  (\ref{defineCG}) gives rise to a fibration sequence
\begin{equation}\label{fibrseq}
 BC\G_\R \rightarrow B \G_\R \rightarrow B\overline{\G(1)}_\R.
 \end{equation}
  
We can save some work by using the following lemma. For a formal power series $p(t) = \sum_{i=0}^{\infty} a_i t^i$ and $q(t) = \sum_{i=0}^{\infty} b_i t^i$, introduce the partial order $ p(t) \leq q(t)$ if and only if $ a_i \leq b_i$ for all $i \in \{0,1,2,...\}$.  

\begin{lem}\label{savetroub}
Suppose that $F \rightarrow E \rightarrow B$ is a Serre fibration such that $H^*(F)$ and $H^*(B)$ are finite dimensional in every degree and $B$ is homotopy equivalent to a connected cell complex such that for every $d\geq 0$, the number of $d$-cells equals $\dim(H^d(B))$.  Then 
\begin{equation}\label{ppp}
 P_t(E) \leq P_t(F) P_t(B)
 \end{equation}
with equality if and only if $H^*(E) \rightarrow H^*(F)$ is surjective.
\end{lem}

\begin{proof}
The $E_2$ page of the Serre spectral sequence is $E_2^{p,q} = H^p( B; H^q(F))$ which is the cohomology of a local system. However, using the cellular decomposition on $B$ we have $E_1^{p,q} = H^p(B) \otimes H^q(F)$.  It follows that $\dim E_{\infty}^{p,q} \leq  \dim ( H^p(B) \otimes H^q(F))$ which implies (\ref{ppp}). Equality only occurs if $E_{\infty}^{p,q}  \cong H^p(B) \otimes H^q(F)$ for all $p,q$ which implies that $E_{\infty}^{0,q}  \cong H^q(F)$ so that $H^*(E) \rightarrow H^*(F)$ is surjective. The converse is simply the Leray-Hirsch Theorem. 
\end{proof} 

By Lemma \ref{homtype1}, the base of (\ref{fibrseq}) is homotopy equivalent to $(S^1)^g$ which admits a cell decomposition satisfying the hypotheses of Lemma \ref{savetroub}.  The Poincar\'e series of  $B \G_\R$ was worked out in (\cite{B} Theorem 6.1)
\begin{eqnarray*}
P_t(B \G_\R;\Z_2) &=& \frac{1-t^{2r}}{(1+t^r)^{a}} \prod_{k=1}^r \frac{(1+t^k)^{2a} (1+t^{2k-1})^{g+1-a}}{(1-t^{2k})^2} \\ 
&=& \frac{(1+t)^g}{1-t} \prod_{k=2}^r \frac{(1+t^{k-1})^a(1+t^k)^a(1+t^{2k-1})^{g+1-a}}{(1-t^{2k})(1-t^{2k-2})}
\end{eqnarray*}
where $a$ is the number of real circles in $(\Sigma,\tau)$. Thus to prove Theorem \ref{surjf} it suffices to show that
\begin{equation}\label{ineqiseq2}
 P_t(B C \G_\R) \leq P_t(B\G_\R) / P_t( B\overline{\G(1)}_\R) = \frac{1}{1-t} \prod_{k=2}^r \frac{(1+t^{k-1})^a(1+t^k)^a(1+t^{2k-1})^{g+1-a}}{(1-t^{2k})(1-t^{2k-2})}
\end{equation}

The short exact sequence (\ref{SESCSR}) determines a fibration sequence $ BS\G_\R \rightarrow BC\G_{\R} \rightarrow B\R^*$ that also satisfies the conditions of Lemma \ref{savetroub} so we find that 

\begin{equation}\label{ineqiseq}
P_t(B C\G_\R) \leq P_t(B S\G_\R) P_t(B \R^*) = P_t(B S\G_\R)/ (1-t).
\end{equation}
Therefore to prove Theorem \ref{surjf} it suffices to prove the following.

\begin{prop}
The cohomology ring $H^*(BS\G_\R;\Z_2)$ has Poincar\'e series 
\begin{equation}\label{formulaforBSGR}
 P_t(BS\G_\R;\Z_2) = \prod_{k=2}^r \frac{(1+t^{k-1})^a(1+t^{k})^a (1+t^{2k-1})^{g+1-a}} {(1-t^{2k})(1-t^{2k-2})}
 \end{equation}
where $a = \pi_0(\Sigma^{\tau})$ and $g$ is the genus of $\Sigma$. 
\end{prop}

\begin{proof}

We refer the reader to (\cite{B} Appendix A) or McLeary (\cite{M} 7.1) for background on the Eilenberg-Moore spectral sequence.

Identify $BS\G_{\R} = BS\G(\hat{g}, n; {\ttau}_1,...,{\ttau}_n)$ from the homotopy pull-back diagram (\ref{hpd}). The associated Eilenberg-Moore spectral sequence  $EM_r^{*,*}$ converges to $H^*(BS\G_\R)$. The second page $EM_2^{*,*}$, equals the cohomology of the differential bi-graded algebra  $ (K^{*,*} \otimes_{R^*} M^*, \delta \otimes 1) $
where 
\begin{itemize}
	\item $(K^{*,*}, \delta)$ is the Koszul-Tate complex (\ref{ktc}),
	\item $M^* = M^{0,*} := \bigotimes_{i=1}^n H^*(BLSU(r)^{{\ttau}_i})$, and
	\item $R^* = R^{0,*}:= \bigotimes_{i=1}^n H^*(BLSU(r)) =   \wedge( \bar{c}_{i,k}) \otimes S(c_{i,k}).$
\end{itemize}

Applying Lemma \ref{aolercubalcroe}, we have an isomorphism of graded $R^*$-modules
$$M^* \cong V \otimes S(c_{i,k})$$
where $V$ is a graded vector space with Poincar\'e series $$P_t(V) = \prod_{k=2}^{r} (1+t^{k-1})^a(1+t^{k})^a (1+t^{2k-1})^{n-a}.$$

We have an isomorphism  $K^{*,*} \otimes_{R^*} M^* \cong \Gamma(z_k) \otimes V \otimes \wedge(x_{i',k}) \otimes S(c_{i,k}) \otimes A$ 
where $\delta(V) = \delta(A) = \delta(c_{i,k}) = \delta(z_k) = 0$ and $\delta(x_{i',k}) = c_{i',k} +c_{1,k}$. Therefore

$$ EM_2^{*,*} = (K^{*,*} \otimes_{R^*} M^*, \delta \otimes 1) \cong  \Gamma(z_k) \otimes V \otimes A \otimes S(c_k). $$

Thus $EM_2^{*,*}$ has Hilbert series with respect to the total grading equal to 
$$  P_t(EM_2^{*,*}) = \prod_{k=2}^r \frac{(1+t^{k-1})^a(1+t^{k})^a (1+t^{2k-1})^{n-a} (1+t^{2k-1})^{2 \hat{g}}}{(1-t^{2k})(1-t^{2k-2})},$$
which equals the right hand side of (\ref{formulaforBSGR}) because $ g = 2 \hat{g} +n - 1$.  It follows then that
$$  P_t(BS\G_\R; \Z_2) \leq \prod_{k=2}^r \frac{(1+t^{k-1})^a(1+t^{k})^a (1+t^{2k-1})^{g+1-a}} {(1-t^{2k})(1-t^{2k-2})}.$$
Since the reverse inequality was already known, the equality (\ref{formulaforBSGR}) holds and the spectral sequence collapses at $EM_2^{*,*}$.

\end{proof}

Consequently both inequalities (\ref{ineqiseq2}) and (\ref{ineqiseq}) are equalities, yielding

\begin{cor}
The cohomology ring $H^*(BC\G_\R;\Z_2)$ has Poincar\'e series 
\begin{equation}\label{formulaforBSGR}
 P_t(BC\G_\R;\Z_2) =   \frac{1}{1-t}\prod_{k=2}^r \frac{(1+t^{k-1})^a(1+t^{k})^a (1+t^{2k-1})^{(g+1-a)}} {(1-t^{2k})(1-t^{2k-2})}
 \end{equation}
where $a = \pi_0(\Sigma^{\tau})$ and $g$ is the genus of $\Sigma$. 
\end{cor}

\subsection{ Characteristic $\neq 2$}

Throughout this subsection, let $\F$ denote a field of odd or zero characteristic. Cohomology will always be taken with coefficients $\F$.  Since $S\G_\R \leq C\G_\R$ has index two, there is a natural identification of $H^*(BC\G_\R; \F)$ with the invariant ring $H^*( BS\G_\R;\F)^{C\G_\R/S\G_\R }$ which we will exploit in our calculation.

The action of $C\G_\R/S\G_\R \cong \Z/2$ on $H^*(BS\G_\R;\F)$ is induced by a group automormorphism of $S\G_\R$ determined by conjugating by an element  $g \in C\G_\R \setminus S\G_R$. Using the real decomposition $E = E' \oplus L$ described in (\ref{decompofrealbundles}), we may choose $g$ to be the gauge transformation which acts trivially on $E'$ and by $-1$ on $L$. This automorphism extends naturally to the diagram (\ref{gcdiag}) and therefore it acts on the spectral sequence (\ref{ComplexEMSS}). This automorphism restricts to an inner automorphism on $Maps(X, U_r)$ and $\prod_{i=1}^n LU_r$. By a theorem of Segal (\cite{S} \S 3), the induced action of $C\G_\R/S\G_\R$ on $BMaps(X, U_r)$ and on $\prod_{i=1}^n BLU_r$ is homotopically trivial so in particular $C\G_\R/S\G_\R$ acts trivially on $K^{*,*}$ and $R^*$. Therefore the only non-trivial contribution to the $C\G_\R/S\G_\R$ action on (\ref{ComplexEMSS}) comes from the action on $H^*(LSU_r^{\tau_i};\F)$ which we investigate next.

\subsubsection{ Cohomology of Real loop groups in odd or zero characteristic}

\begin{prop}\label{oddrealloops}
Let $r$ be a positive integer and let $r = 2r'+1$ or $r=2r'$ depending on whether $r$ is even or odd.
We have isomorphism
$$ H^*(BLSO(2r'+1);\F) \cong \wedge( \bar{p}_1,...,\bar{p}_{r'} ) \otimes S(p_1,...,p_{r'})$$
and 
$$ H^*(BLSO(2r');\F) \cong \wedge(\bar{p}_1,..., \bar{p}_{r'-1}, \bar{e}_{r'}) \otimes S(p_1,...,p_{r'-1}, e_{r'}) .$$
with degrees $|p_k| = 4k$, $|\bar{p}_k| = 4k-1$, $|e_{r'}| = 2r'$, and $|\bar{e}_{r'}| = 2r'-1$. 
In the even rank case denote $p_{r'} = e_{r'}^2$ and $\bar{p}_{r'} = 2 \bar{e}_{r'}e_{r'}$ for convenience. The inclusion $\iota: LSO(r) \hookrightarrow LSU(r)$ induces a morphism on cohomology from $H^*(BLSU(r)) = \wedge(\bar{c}_2,..., \bar{c}_r) \otimes S(c_2,...,c_r)$ to $H^*(BLSO(r))$ satisfying
\begin{itemize}
 \item $\iota^*(\bar{c}_{2k-1}) = \iota^*(c_{2k-1}) = 0$,  for all $k$
 \item $\iota^*(c_{2k}) = p_k$  and $\iota^*(\bar{c}_{2k}) = \bar{p}_k$ for all $k$,
 \end{itemize}
The conjugation action of $LO(r)/LSO(r) \cong \Z/2$ on $H^*(BLSO(r);\F)$ is trivial on generators $p_i$ and $\bar{p}_i$ for all $i$ and by $-1$ on $e_{r'}$ and $\bar{e}_{r'}$.
\end{prop}

\begin{proof}
The formulas for $H^*(BLSO(r);\F)$ can be deduced from (\cite{KMN} Theorem 2) using the well known fact that $H^*(BSO(r);\F)$ is a polynomial ring generated by Pontryagin classes and (if $r$ is even) the Euler class. Using the identification $BLSO(r) \cong LBSO(r)$,  we get an evaluation map $ev: S^1 \times LBSO(r) \rightarrow BSO(r)$. The generators are defined by $p_i := \int_{pt} ev^*(p_i)$ and $\bar{p}_i := \int_{S^1} ev^*(p_i)$  where $\int$ denotes the slant product with respect to homology class in $H_*(S^1;\F)$ and $e_i, \bar{e}_i$ and $c_i, \bar{c}_i$ are defined similarly.  The formula for $i^*$ follows from the well known relationships between Chern classes, Pontryagin classes, and Euler classes described in \cite{MiSt}.  We refer to \cite{B} \S 4 where this construction is laid out in greater detail.     
\end{proof}

\begin{cor}\label{oddrealloops2}
The invariant subring of the $LO(r)/LSO(r) \cong \Z/2$ automorphism described above satisfies
$$  H^*(BLSO(2r');\F)^{\Z/2} \cong H^*(BLSO(2r'+1);\F)^{\Z/2} \cong H^*(BLSO(2r'+1);\F) \cong \wedge( \bar{p}_1,...,\bar{p}_{r'} ) \otimes S(p_1,...,p_{r'})  $$
with all isomorphism induced by the obvious inclusions. The induced map $H^*(BLSU_r;\F) \rightarrow H^*(BLSO(r);\F)^{\Z/2}$ sends
\begin{eqnarray*}
&  \iota^*(\bar{c}_{2k-1}) = \iota^*(c_{2k-1}) = 0, & \text{ for all $k$}\\
 & \iota^*(c_{2k}) = p_k \text{ and } \iota^*(\bar{c}_{2k}) = \bar{p}_k & \text{for all $k $.}
 \end{eqnarray*}

\end{cor}

\begin{prop}\label{abccohom}
If $r$ is odd, then
$$ H^*(BLSU(r)^{\tau_\alpha};\F) \cong  H^*(BLSU_{r}^{\tau_\beta}; \F)  \cong H^*(BLSU_{r}^{\tau_\gamma}; \F)  \cong H^*(BLSO(r);\F). $$ 
If $r$ is even, then 
\begin{eqnarray*}
 H^*(BLSU(r)^{\tau_\alpha};\F)  & \cong & H^*(BLSO(r);\F).  \\
 H^*(BLSU(r)^{\tau_\beta}; \F) & \cong & H^*(BLSO(r-1);\F) \\
 H^*(BLSU(r)^{\tau_\gamma};\F) & \cong & H^*(BLSO(r);\F)^{\Z/2}.
\end{eqnarray*}
In all three cases, the homomorphism $H^*(BLSU_r;\F) \rightarrow H^*(BLSU_r^{\tau};\F) $ agrees with the homomorphisms described in Propositions \ref{oddrealloops} and \ref{oddrealloops2} on generators, up to multiplication by a non-zero scalar.
\end{prop}

\begin{proof}
In case $\alpha$ we have an equality $LSU(r)^{\tau_\alpha} = LSO(r)$ so there is nothing to prove.

The cases $\beta$ and $\gamma$ can be identified with twisted loop groups, and their cohomology has already been calculated in \cite{B2} for characteristic greater than $r$. The remaining odd primes can be dealt with as follows. We treat only the case $\gamma$ in detail since $\beta$ is dealt with similarly. 

First note that since  $H^*(BLSU(r)^{\tau_c}; \Q) \rightarrow H^*(BLSO(r); \Q)^{\Z/2} $ is known to be surjective from  \cite{B2} and  $H^*(BLSO(r); \Z)$ does not contain $p$ torsion for any odd $p$, it follows that 
\begin{equation}\label{invoke soon}
H^*(BLSU(r)^{\tau_c}; \F) \rightarrow H^*(BLSO(r); \F)^{\Z/2} 
\end{equation} 
is surjective. 
We have a short exact sequence $1 \rightarrow \Omega SU(r) \rightarrow LSU(r)^{\tau_\gamma} \rightarrow SU(r) \rightarrow 1$ which gives rise to a fibration sequence $SU(r) \rightarrow BLSU(r)^{\tau_\gamma} \rightarrow BSU(r)$, where we have employed the homotopy equivalence $B \Omega SU(r) \cong SU(r)$. The Serre spectral sequence $(E_k, \delta_k)$, with $ \delta_k: E_k^{p,q}\rightarrow E_k^{p+k, q-k+1}$ converges to $BLSU(r)^{\tau_\gamma} $ and has $E_2^{p,q} = H^q(SU(r);\F) \otimes H^p(BSU(r))$ where $ H^*(SU(r);\F) \otimes H^*(BSU(r)) \cong  \wedge(\bar{c}_2,...,\bar{c}_r) \otimes S(c_2,...,c_r)$. By the surjectivity of (\ref{invoke soon}) the even generators $\bar{c}_{2i}$ survive to $E_{\infty}$ for all  $i$.  

We claim that the odd generators $\bar{c}_{2i+1}$ are all transgressive, meaning that $\delta_k(\bar{c}_i) = 0$ for $k < 2i$. Since $E_2^{*,*}$ is torsion free, it suffices to prove this for $\F = \Q$, when we know that (\ref{invoke soon}) is an isomorphism. Since $\bar{c}_2$ survives to infinity, the only class that can kill $c_3$ is $\bar{c}_3$. Since we know that $c_3$ is killed (when $\F = \Q$), it follows that $\bar{c}_3$ is transgressive, so $\delta_{6} (\bar{c}_3) = \lambda c_3$ for some non-zero scalar $\lambda$, hence $\delta_k(\bar{c}_3) = 0$ for $k <6$. By induction, this implies that the only class that can kill $c_5$ is $\bar{c}_5$ and so on.

Therefore, we know that for all $i$, $\delta_{4i+2}(\bar{c}_{2i+1}) = \lambda_i c_{2i+1}$ for some nonzero integer $\lambda_i$. It remains to show that the $\lambda_i$ is not divisible by any odd prime $p$. If it were, that would mean $\bar{c}_{2i+1}$ survives to $E_{\infty}$.  But this is not true by the following argument. Consider the family of automorphisms of $LSU(r)^{\tau_\gamma} \leq Maps(S^1, SU(r))$ obtained by rotating the the domain circle. Since this is a path connected family, they all act by isotopies on $BLSU(r)^{\tau_\gamma} $ and hence act trivially on cohomology.  However, if we rotate by 180 degrees, this has the effect on the fibre of (\ref{invoke soon}) of complex conjugating the matrix entry-wise. In terms of the cohomology ring $\wedge(\bar{c}_2,...,\bar{c}_r)$ this sends $\bar{c}_{2i} \mapsto \bar{c}_{2i}$ and $\bar{c}_{2i+1} \mapsto -\bar{c}_{2i+1}$. It follows that $\bar{c}_{2i+1}$ is not the restriction of a class in $H^*(BSU(r)^{\tau_\gamma};\Z_p)$ for $p$ odd hence it does not survive to $E_{\infty}$.  

The argument for case $\beta$ is similar, except it is only the primitive $ \bar{e}_r$ of the Euler class that must be shown to be transgressive and the rotation automorphism must also incorporate the twist coming from the Moebius bundle defining $L SU_r^{\tau_\beta}$. Lifting a $360$ degree rotation of the circle to the Moebius bundle determines an orientation reversal of the fibres and sends $\bar{e}_r$ to $-\bar{e}_r$ and the argument goes through as before.
\end{proof}

\subsubsection{Cohomology of $B C \G_\R$ in odd or zero characteristic}

\begin{thm}\label{BigLongThm}
Let $\F$ be a field of odd or zero characteristic. 

\textbf{Case 1} If the rank $r = 2r'+1$ is odd, then the Poincar\'e series  equals
\begin{equation}\label{oddoddpp}
P_t(BC\G_\R;\F) = P_t(BS\G_\R) =  \prod_{k'=1}^{r'} \frac{(1+t^{4k'-1})^{g} (1+t^{4k'+1})^{g}}{(1-t^{4k'})^2}
\end{equation}
which depends only on the rank $r$ and degree $g$.

\textbf{Case 2} If the rank $r= 2r'$ is even, then the Poincar\'e series factors 
$$ P_t(C \G_{\R};\F)  = F_t G_t $$
where 
$$F_t =   \prod_{k''=1}^{r'-1} \frac{(1+t^{4k''-1})^{g} (1+t^{4k''+1})^{g}}{(1-t^{4k''})^2}$$
depends only on the rank $r$ and the genus $g$ and $G_t$ is defined case by case below.

Let $a$ be the number real circles of $(\Sigma, \tau)$ of which $b$ are odd and $c$ are even with respect to $(E,\ttau)$. Then

\begin{itemize}
\item If $a=0$, then
$$ G_t =   \frac{(1+t^{2r-1})^{g}} {(1-t^{2r})}  $$

\item If $a > c \geq 0$ and $\Sigma \setminus \Sigma^{\tau}$ is connected then
$$ G_t = \frac{ (1+t^{r-1})^{c}(1+t^r)^{c} +(1-t^{r-1})^{c} (1-t^r)^{c} }{2} (1+t^{2r-1})^{g-c-1}$$

\item If $a = c > 0$ and $\Sigma \setminus \Sigma^{\tau}$ is connected then
$$ G_t = \frac{ (1+t^{r-1})^{c}(1+t^r)^{c-1} +(1-t^{r-1})^{c} (1-t^r)^{c-1} }{2(1-t^r)} (1+t^{2r-1})^{g-c}$$

\item If $a> c = 0$ and $\Sigma \setminus \Sigma^{\tau}$ is disconnected then
$$ G_t = \frac{(1+t^{2r-1})^{g}}{1-t^{2r-2}}$$

\item If $a>c > 0$, $c$ is odd, and $\Sigma \setminus \Sigma^{\tau}$ is disconnected then
$$ G_t =  \frac{ (1+t^{r-1})^{c} (1+ t^r)^{c} + (1- t^{r-1})^{c} (1- t^r)^{c}}{2} (1+t^{2r-1})^{g-c-1}  $$

\item If $a> c > 0$, $c$ is even, and $\Sigma \setminus \Sigma^{\tau}$ is disconnected then
$$ G_t =   \Big(  \frac{t^{r-1} (t^{r-1}+ t^r)^{c+1}}{(1-t^{2r-2})}+ \frac{(1+t^{r-1})^{c} (1+ t^r)^{c} + (1- t^{r-1})^{c} (1- t^r)^{c}}{2}\Big) (1+t^{2r-1})^{g-c-1}$$

\item If $a = c > 0$ and $\Sigma \setminus \Sigma^{\tau}$ is disconnected then
$$  G_t = \frac{ (1+t^{r-1})^c(1+t^r)^{c-1}+ (1-t^{r-1})^{c} (1-t^r)^{c-1}}{2(1-t^{2r})}(1+t^{2r-1})^{g-c}$$

\end{itemize}

\end{thm}

\begin{proof}

Denote $BS\G_{\R} = BS\G(\hat{g}, n; {\ttau}_1,...,{\ttau}_n)$ from the homotopy pull-back diagram (\ref{hpd}). The associated Eilenberg-Moore spectral sequence  $EM_r^{*,*}$ converges to $H^*(BS\G_{\R});\F)$. The second page $EM_2^{*,*}$, equals the cohomology of the  differential bi-graded algebra  $ (K^{*,*} \otimes_{R^*} M^*, \delta \otimes 1) $ described in (\ref{ComplexEMSS})
where 
\begin{itemize}
	\item $(K^{*,*}, \delta)$ is the Koszul-Tate complex (\ref{ktc}),
	\item $M^* = M^{0,*} := \bigotimes_{i=1}^n H^*(BLSU(r)^{{\ttau}_i})$, and
	\item $R^* = R^{0,*}:= \bigotimes_{i=1}^n H^*(BLSU(r)) =   \wedge( \bar{c}_{i,k}) \otimes S(c_{i,k}).$
\end{itemize}

\textbf{Case 1: $r = 2r'+1$ is odd}

In this case $C\G_\R \cong S\G_\R \times \R^*$, so $H^*(BC\G_\R;\F) \cong  H^*(BS\G_\R \times B\R^*;\F) \cong H^*(BS\G_\R;\F)$, because $B\R^* = \R P^{\infty}$ is $\F$-acyclic. So it suffices to compute $H^*(BS\G_\R;\F)$.

Recall that we have index sets $i \in \{1,...,n\}$, $i' \in \{2,...,n\}$, $k \in \{2,...,r \}$ and introduce a further index set $k' \in \{1,...,r'\}$. We have 
$$
K^{*,*} \otimes_{R^*} M^* =   K^{*,*} := \Gamma(z_k) \otimes \wedge(x_{i',k}) \otimes \wedge(\bar{p}_{i,k'}) \otimes S(p_{i,k'}) \otimes A
$$
with bidegrees and differentials

\begin{tabular}{|c|c|c|}
	\hline
generator   & bi-degree  & $\delta$-derivative \\
\hline
$\bar{p}_{i,k'}$ & $(0,4k'-1)$ & $0$ \\
$p_{i,k'}$ & $(0,4k')$ & $0$\\
$x_{i',2k'}$ & $(-1,4k')$ & $p_{i',k'}-p_{1,k'}$ \\
$x_{i',2k'+1}$ & $(-1,4k'+2)$ & $0$ \\
$z_{2k'}$ & $(-1,4k'-1)$ & $\bar{p}_{1,k'}+...+\bar{p}_{n,k'} $\\
$z_{2k'+1}$ & $(-1,4k'+1)$ & $0$\\
\hline	

\end{tabular}

Taking cohomology yields
\begin{equation}\label{oddcase}
EM_2^{*,*} = \Gamma(z_{2k'+1}) \otimes \wedge(x_{i',2k'+1}) \otimes \frac{ \wedge (\bar{p}_{i,k'})}{(\bar{p}_{i,1} +...+\bar{p}_{i,r'})} \otimes S(p_{k'}) \otimes A
\end{equation}
where we abuse notation as usual and denote cohomology classes by representative cocycles.

Over the rational cofficients, $\Gamma(z_{2k'-1})  \cong S(z_{2k'-1})$ so the bigraded ring $EM_2^{*,*}$ is generated by elements in the $-1$ and $0$ columns, which implies that $EM_2^{*,*} = EM_{\infty}^{*,*}$.  By the universal coefficient theorem, the spectral sequence must collapse for all fields under consideration. Therefore (\ref{oddcase}) is isomorphic to an associated graded ring of $H^*(BS\G_\R; \F)$, yielding (\ref{oddoddpp}).

\textbf{Case 2: $r=2r'$ is even} 

We suppose that the first $0\leq a \leq n$ boundary circles are real circles, and that the first $0 \leq b \leq a$ have SW class one and the remaining have zero.  

We introduce another index set $k''\in \{1,...,r'-1\}$. Applying Proposition \ref{abccohom} we have
\begin{eqnarray*} K^{*,*} \otimes_{R^*} M^* & =&   \Gamma(z_k) \otimes \wedge ( x_{i',k}) \otimes \wedge (p_{i,k''}) \otimes S(p_{i,k''}) \otimes A \\
& &  \otimes \wedge (\bar{e}_{b+1}, ..., \bar{e}_a) \otimes S(e_{b+1}, ..., e_a) \otimes \wedge (\bar{p}_{a+1,r'}, ...,\bar{p}_{n,r'}) \otimes S(p_{a+1, r'},...,p_{n,r'})
\end{eqnarray*}
with bidegrees and differentials

\begin{tabular}{|c|c|c|}
	\hline
generator   & bi-degree  & $\delta$-derivative \\
\hline
$\bar{p}_{i,k'}$ & $(0,4k'-1)$ & $0$ \\
$p_{i,k'}$ & $(0,4k')$ & $0$\\
$e_i$ & $(0,r)$ &$0$\\
$\bar{e}_i$ & $(0,r-1) $ &$0$\\
$x_{i',2k''}$ & $(-1,4k')$ &   $ p_{i',k''} -p_{1,k''}$  \\
$x_{i',2r'}$ & $(-1,4r')$ &   $ \begin{cases}  p_{i',r'} -p_{1,r'} & \text{if $b=0$} \\ p_{i',r'} & \text{if $i' > b >0$} \\ 0 & \text{if $b \geq i'$} \end{cases}$ \\
$x_{i',2k''+1}$ & $(-1,4k''+2)$ & $0$ \\
$z_{2k''}$ & $(-1,4k''-1)$ & $ \bar{p}_{1,k''}+...+\bar{p}_{n,k''} $\\
$z_{2r'}$ & $(-1,4r'-1)$ & $  \bar{p}_{b+1,r'}+...+\bar{p}_{n,r'} $\\
$z_{2k''+1}$ & $(-1,4k''+1)$ & $0$\\
\hline	
\end{tabular}

where recall we denote $e_i^2 = p_{i, r'}$ and $2\bar{e}_i e_i = \bar{p}_{i,r'}$ for $i \in \{b+1,...,a\}$. 

This decomposes as a tensor product of dgas, 
$$ K^{*,*} \otimes_{R^*} M^* = S^{*,*}\otimes T^{*,*}  $$
where
\begin{eqnarray*}
S^{*,*} & =& \Gamma( z_k| k<r) \otimes \wedge(x_{i',k}| k<r) \otimes \wedge(p_{i,k''}) \otimes S(p_{i,k''}) \otimes A\\
T^{*,*} &=& \Gamma(z_r) \otimes \wedge(x_{i',r}) \otimes \wedge (\bar{e}_{b+1}, ..., \bar{e}_a) \otimes S(e_{b+1}, ..., e_a) \otimes \wedge (\bar{p}_{a+1,r'}, ...,\bar{p}_{n,r'}) \otimes S(p_{a+1, r'},...,p_{n,r'})
\end{eqnarray*}
so we may use the Kunneth formula
$$ EM_2^{*,*} = H(K^{*,*} \otimes_{R^*} M^*) = H(S^{*,*}) \otimes H(T^{*,*}) .$$

Note that $S^{*,*}$ is independent of $a$ or $b$ with cohomology easily computed

$$  H(S^{*,*}) \cong \Gamma(z_{2k''+1})  \otimes \frac{\wedge (\bar{p}_{i,k''})}{(\bar{p}_{1,k''}+...+\bar{p}_{n,k''})}\otimes \wedge (x_{i',2k''+1}) \otimes S(p_{k''}) \otimes A. $$
and Poincar\'e series
\begin{eqnarray*}
P_t( H(S^{*,*})) &=&  \Big( \prod_{k''=2}^{r'-1} \frac{(1+t^{4k''-1})^{n-1} (1+t^{4k''+1})^{n-1}}{(1-t^{4k''})^2}\Big) \prod_{k=2}^{2r'} (1+t^{2k-1})^{2 \hat{g}}\\
&=&  (1+t^{2r-1})^{2 \hat{g}}  \prod_{k''=2}^{r'-1} \frac{(1+t^{4k''-1})^{g} (1+t^{4k''+1})^{g}}{(1-t^{4k''})^2}
\end{eqnarray*}

Our next task is to calculate the Betti numbers of $H(T^{*,*})$. Since we are ultimately interested in $H(T^{*,*})$ as a graded ring with $\Z/2$-action, we will consider $P_t(H(T^{*,*}))$ with coefficients lying in the ring of characters for $\Z/2$ where $1$ denotes the character of the trivial irrep and  $\chi$ of the non-trivial irrep of $\Z/2$.  

To calculate the Betti numbers of $H(T^{*,*})$ we use a filtration of $T^{*,*}$ and consider the associated trigraded spectral sequence $E^{*,*,*}_*$ converging to $H(T^{*,*})$.  Consider the filtration by bigraded dga ideals $$T^{*,*} = F^0 \supset F^1 \supset ... \supset F^{a-b+1} = 0$$ where $F_k := \wedge^{\geq k}(\bar{e}_{b+1},..., \bar{e}_a) T^{*,*}$. Taking subquotients determines a differential tri-graded algebra $E_1^{*,*,*}$ such that $E^{*,*,k}_1 = F^k/F^{k+1}$ and $\delta: E_1^{p,q,k} \rightarrow E_1^{p+1,q,k}$. If we ignore the third grading, then there is an isomorphism of bigraded algebras $E_1^{*,*,*} \cong T^{*,*}$, but it does not respect differentials.
For $\bar{e}_I \in \wedge^k(\bar{e}_{b+1}, ..., \bar{e}_a)$, the differential on $E^{*,*,*}_1$ is determined by the identities
$$ \delta(\bar{e}_I z_r) = \bar{e}_I (\bar{p}_{a+1} +...+\bar{p}_{n,r'})$$
$$ \delta(\bar{e}_I x_{i',r}) = \begin{cases} \bar{e}_I( p_{i',r'} - p_{1,r'} )& \text{if $a=0$} \\  \bar{e}_I p_{i',r'} & \text{if $i' >a > 0$} \\ 
\bar{e}_I e_{i',r'}^2 & \text{if $a \geq i'> b$} \\ 
0 &  \text{if $b \geq i' $} \end{cases} .$$
Define $E_2^{*,*,*} := H(E_1^{*,*,*} ,\delta)$. We consider three different cases in order of increasing difficulty.

\textbf{Case i:} $a=0$

In this case $F_1 =0$ and the filtration is trivial. We have
$$T^{*,*} =  \Gamma(z_r) \otimes \wedge(x_{i',r}) \otimes \wedge (\bar{p}_{1,r'}, ...,\bar{p}_{n,r'}) \otimes S(p_{1, r'},...,p_{n,r'})$$
and $$H(T^{*,*}) = \frac{\wedge (\bar{p}_{1,r'}, ...,\bar{p}_{n,r'})}{(\bar{p}_{1,r'}+...+\bar{p}_{n,r'} )} \otimes S(p_{r'}).$$
so $$ P_t(H(T^{*,*})) = \frac{(1+t^{2r-1})^{n-1}} {(1-t^{2r})} .$$

\textbf{Case ii} $ 0< a<n$
In this case
$$  E_1^{*,*,*} =  \Gamma(z_r) \otimes \wedge(x_{i',r}) \otimes \wedge (\bar{e}_{b+1}, ..., \bar{e}_a) \otimes S(e_{b+1}, ..., e_a) \otimes \wedge (\bar{p}_{a+1,r'}, ...,\bar{p}_{n,r'}) \otimes S(p_{a+1, r'},...,p_{n,r'}). $$
If $b >0$ we get
$$ E_2^{*,*,*}  :=  H(E_1^{*,*,*} ,\delta) \cong  \wedge( x_{2,r},..., x_{b,r}) \otimes \wedge^{k} (\bar{e}_{b+1}, ..., \bar{e}_a) \otimes \frac{S(e_{b+1}, ..., e_a)}{(e_{b+1}^2,...,e_a^2)} \otimes \frac{ \wedge (\bar{p}_{a+1,r'}, ...,\bar{p}_{n,r'})}{(\bar{p}_{a+1,r'}+ ...+\bar{p}_{n,r'})} $$
and if $b=0$ we get 
$$ E_2^{*,*,*} = \wedge^{k} (\bar{e}_{1}, ..., \bar{e}_a) \otimes S(e_1) \otimes \frac{S(e_{2}, ..., e_a)}{(e_{2}^2,...,e_a^2)} \otimes \frac{ \wedge (\bar{p}_{a+1,r'}, ...,\bar{p}_{n,r'})}{(\bar{p}_{a+1,r'}+ ...+\bar{p}_{n,r'})} .$$

Notice that in both cases, the classes in $E_2^{*,*,*}$ are represented by cycles in $T^{*,*}$. It follows that $E_2^{*,*,*} = E_{\infty}^{*,*,*}$ and that we get an isomorphism of bigraded vector spaces $H(T^{p,q})  = \sum_k E_2^{p,q,k}$ yielding
$$P_t(H(T^{*,*})) = (1+ \chi t^{r-1})^{a-b} (1+\chi t^r)^{a-b}(1+t^{2r-1})^{n-a+b-2} $$
if $b>0$ and
$$P_t(H(T^{*,*})) = \frac{(1+\chi t^{r-1})^{a} (1+\chi t^r)^{a-1}(1+t^{2r-1})^{n-a-1}}{1-t^r} $$
if $b=0$. Furthermore, $H(T^{*,*})$ is generated as ring by elements lying in columns $(0,*)$ and $(1,*)$.

\textbf{Case iii:} $a=n$

If $b=n$ then $T^{*,*} =  \Gamma(z_r) \otimes \wedge(x_{i',r})$ and the coboundary map is trivial so $T^{*,*} = H(T^{*,*})$. 

If $b \neq n$ then
$$ E_1^{*,*,*} =  \Gamma(z_r) \otimes \wedge(x_{i',r}) \otimes \wedge (\bar{e}_{b+1}, ..., \bar{e}_n) \otimes S(e_{b+1}, ..., e_n). $$
If $n>b>0$, then
$$ E_2^{*,*,*} \cong \Gamma(z_r) \otimes \wedge(x_{2,r},..., x_{b}) \otimes \wedge (\bar{e}_{b+1}, ..., \bar{e}_n)  \otimes  \frac{S(e_{b+1}, ..., e_n)}{(e_{b+1}^2,...,e_n^2)} $$
and if $b=0$, then 
$$ E_2^{*,*,*} \cong \Gamma(z_r) \otimes \wedge(\bar{e}_{1}, ..., \bar{e}_n)  \otimes S(e_1) \otimes \frac{S(e_{2}, ..., e_n)}{(e_{2}^2,...,e_n^2)}. $$
We must now calculate $E_3^{*,*,*}$. The boundary map for $E_2^{*,*,*}$ is determined by $\delta(z_r) = \bar{p} := \bar{e}_{b+1} e_{b+1}+...+ \bar{e}_n e_n$. Observe that $\bar{p}^2=0$. 

Define $$S :=  \begin{cases} \wedge(\bar{e}_{b+1}, ..., \bar{e}_n)  \otimes  \frac{S(e_{b+1}, ..., e_n)}{(e_{b+1}^2,...,e_n^2)} & \text{if $b>0$}\\  \wedge(\bar{e}_{1}, ..., \bar{e}_n)  \otimes S(e_1) \otimes \frac{S(e_{2}, ..., e_n)}{(e_{2}^2,...,e_n^2)} & \text{ if $b=0$} \end{cases}$$  
and denote $Ann(S):= \{ s \in S|ps=0 \}$ the annihilator of $\bar{p}$.  Consider the chain complex

$$ ... \rightarrow^{\delta} S z_r^{[3]} \rightarrow^{\delta} S z_r^{[2]} \rightarrow^{\delta} S z_r \rightarrow^{\delta} S  $$
where the boundary map is $\delta (s z_r^{[d]}) = s \bar{p} z_r^{[d-1]}$ for any $s \in S$.  It is clear from this point of view that
\begin{eqnarray}\label{E3a}
E_3^{*,*,*} & \cong & \wedge(x_{2,r},..., x_{b,r}) \otimes \Big(  \frac{S}{\bar{p}S}  \oplus (\bigoplus_{d=1}^{\infty}  H(S) z_r^{[d]} )  \Big) \\ & =& \wedge(x_{2,r},..., x_{b,r}) \otimes \Big( \frac{ S + Ann(\bar{p}) \otimes \Gamma(z_r) }{\bar{p}S} \Big)
\end{eqnarray}
where $H(S) := Ann(\bar{p})/\bar{p}S$.  Since these generators lift to cycles in $T^{*,*}$ it follows that $E_3^{*,*,*} =E_{\infty}^{*,*,*}$ and that 
\begin{equation}\label{E3}
H(T^{*,*}) = \sum_k E_3^{*,*,k}.
\end{equation}
Furthermore, note that if $\F = \Q$, then $\Gamma(z_r) = S(z_r)$, so we can choose generators of $H(T^{*,*})$ lying in columns $(0,*)$ and $(-1,*)$.

\begin{lem}
$$H(S) = \begin{cases} \F \{\bar{e}_{b+1}, e_{b+1}\} \otimes ... \otimes \F \{\bar{e}_{n}, e_{n}\} & \text{ if $b>0$} \\ 0 & \text{ if $b=0$.}  \end{cases} $$ 
.
 \end{lem}
\begin{proof}  
The group $G :=  \{\pm 1\}^{\{b+1,.., n \}} \cong  (\Z/2)^{n-b}$ acts by automorphisms on $S$ where $g \in G$ acts by 
$$ g \cdot e_i  = g(i) e_i$$ 
$$g \cdot \bar{e}_i  = g(i) \bar{e}_i$$
Since $G$ stabilizes $\bar{p}$, the action restricts to both $Ann(\bar{p})$ and $\bar{p}S$ and thus descends to $H(S)$. Let $g_q \in G$ be the element $$g_q(i) = \begin{cases} 1 & \text{ if $i \neq q$} \\ -1 & \text{ if $i = q$}  \end{cases}$$
and denote $S^{g_q}$ the subring of $g_q$ invariants. Let $\bar{p}_q = \bar{e}_qe_q$. 

Assume that $n> b>0$. Then
$$S^{g_n} = \wedge ( \bar{e}_{b+1},..., \bar{e}_{n-1}) \otimes \frac{S(e_{b+1},...,e_{n-1})}{(e_{b+1}^2,...,e_{n-1}^2)} \otimes \wedge(\bar{p}_n) \cong   \frac{S(e_{b+1},...,e_{n-1})}{(e_{b+1}^2,...,e_{n-1}^2)} \otimes \wedge(\bar{p}) $$  
where in the last step we have changed variables to replace $\bar{p}_n$ with $ \bar{p} = \bar{e}_{b+1} e_{b+1}+...+ \frac{1}{2} \bar{p}_n$. It is clear then that $Ann(\bar{p})^{g_n} = (\bar{p}S)^{g_n}$ so $H(S)^{g_n}$=0. Similarly, $H(S)^{g_q}=0$ for all $q \in \{b+1,...,n\}$. 
It follows that every non-zero element of $H(S)$ transforms by $-1$ under $g_q$ for every $q \in \{b+1,...,n\}$. The corresponding weight space in $S$ is
$$ S_{(-1,..,-1)} = \F \{\bar{e}_{b+1}, e_{b+1}\} \otimes ... \otimes \F \{\bar{e}_{n}, e_{n}\} $$  
which is annihilated by $\bar{p}$, so we have $H(S) = S_{(-1,...,-1)} $.

Next assume that $b=0$. Applying the analogous argument, we deduce that $H(S) = H(S)_{(-1,...,-1)}$. But now
$$ S_{(-1,..,-1)} = \F \{\bar{e}_{1} e_1^{2k}, e_{1}^{2k+1} | k \geq 0\} \otimes  \F \{\bar{e}_{2}, e_{2}\} ... \otimes \F \{\bar{e}_{n}, e_{n}\}. $$
Clearly $ Ann(\bar{p})_{(-1,...,-1)} = Ann(\bar{e}_{1} e_1)_{(-1,...,-1)}  = (\bar{e}_{1} e_1) S = \bar{p} S$ so we conclude $H(S)=0$. 

\end{proof}

Next observe that 
\begin{eqnarray*} 
P_t(S) &=& \frac{1}{t^{2r-1}} P_t(p S) + P_t(Ann(S))\\
&=&  \frac{1+t^{2r-1}}{t^{2r-1}} P_t(p S) + P_t(H(S))\\
\end{eqnarray*}
so
$$ P_t(S/pS) = P_t(S) -P_t(pS) = \frac{P_t(S) - P_t(H(S)t^{2r-1}}{1+t^{2r-1}} .$$
If $n> b>0$, then combining with (\ref{E3a}) and (\ref{E3}) yields 
\begin{eqnarray*}
P_t(H(T^{*,*})) & =& (1+t^{2r-1})^{b-1} ( \frac{t^{2r-2}}{1-t^{2r-2}} P_t(H(S)) + P_t(S/pS) ) \\
 &=& (1+t^{2r-1})^{b-1} \Big(  \frac{t^{2r-2}(\chi t^{r-1}+\chi t^r)^{n-b}}{1-t^{2r-2}}+ \\ && \frac{(1+\chi t^{r-1})^{n-b} (1+\chi t^r)^{n-b} -t^{2r-1}(\chi t^{r-1}+\chi t^r)^{n-b} }{1+t^{2r-1}}  \Big) \\
 &=& (1+t^{2r-1})^{b-2} \Big(  \frac{(t^{2r-2} + t^{2r-1}) (\chi t^{r-1}+\chi t^r)^{n-b}}{(1-t^{2r-2})}+ (1+\chi t^{r-1})^{n-b} (1+\chi t^r)^{n-b}   \Big)
 \end{eqnarray*}
and if $b=0$ we get 
\begin{eqnarray*}
P_t(H(T^{*,*})) &=&  \frac{P_t(S)}{1+t^{2r-1}}  \\
&=&   \frac{ (1+ \chi t^{r-1})^n(1+ \chi t^r)^{n-1}}{(1+t^{2r-1})(1- \chi t^r)} \\
\end{eqnarray*}

In all  cases we see that if $\F = \Q$, then $EM_{2}^{*,*} = H(R^{*,*})\otimes H(T^{*,*})$ is generated by elements in the $(0,*)$ and $(-1,*)$ columns, which implies $EM_2^{*,*}= EM_{\infty}^{*,*} $. The case for general $\F$ in odd characteristic follows by the universal coefficient theorem. This means in particular that
$$ P_t( BS\G_\R) = P_t( EM_2^{*,*}) = P_t(H(S^{*,*}))P_t(H(T^{*,*})). $$

Finally we must consider the action of $C \G_\R/S\G_\R \cong \Z/2$ on $H^*(BS\G_\R)$. Since $\Z/2$ is semisimple over $\F$ (terminology) we have an isomorphism $H^*(BS\G_\R) \cong EM_2^{*,*}$ as graded $\Z/2$-representations.  The action on $EM_2^{*,*}$ sends $e_i \rightarrow -e_i$ and $\bar{e}_i \mapsto -\bar{e}_i$ for all $i \in \{b+1,...,a\}$ and acts trivially on the remaining generators.  This action is trivial on $R^{*,*}$, so we have
$$(EM_2^{*,*})^{\Z/2} = H^*(S^{*,*}) \otimes H(T^{*,*})^{\Z/2}. $$
and 
$$ P_t(BC\G_\R)= \frac{1}{2}P_t(H(S^{*,*}))\Big( P_t^{\chi =1} (H(T^{*,*})) + P_t^{\chi =-1} (H(T^{*,*}))  \Big) .$$

Finally, we define
$$ F_t = P_t(H(S^{*,*})) \frac{1}{(1+ t^{2r-1})^{2\hat{g}}}$$
$$ G_t =   \frac{(1+ t^{2r-1})^{2\hat{g}}}{2} \Big( P_t^{\chi =1} (H(T^{*,*})) + P_t^{\chi =-1} (H(T^{*,*}))  \Big) $$
\end{proof}

\begin{cor}\label{Bettis2r-2}
Let $(\Sigma,\tau)$ be a real curve with $a$ real circles and let $\xi$ be a real line bundle for which $c$ real circles are even and let $r$ be even. Then the polynomial $G_t$ appearing in Theorem \ref{BigLongThm} satisfies
$$ G_t = \beta_{2r-2}t^{2r-2} + \beta_{2r-1} t^{2r-1} +O(t^{2r}) $$ 
where $$ \beta_{2r-2} = \begin{cases}   { c \choose 2}  & \text{ if $c \geq 1$} \\ 1 & \text{ if $c = 0$ and $a \geq 1$} \\ 0 & \text{ if $c=a=0$}  \end{cases} $$
$$ \beta_{2r-1} = \begin{cases}  g+ c^2-c-1 & \text{ If $a > c > 0$ } \\ 
 g-1  & \text{ if $a> c = 0$ and $\Sigma \setminus \Sigma^{\tau}$ is connected}  \\ 
 g & \text{ if $a> c = 0$ and $\Sigma \setminus \Sigma^{\tau}$ is disconnected}  \\ 
 g+ c^2-c &  \text{ if $a= c > 0$ and $\Sigma \setminus \Sigma^{\tau}$ is connected}  \\ 
 g+ c^2-2c &  \text{ if $a= c > 0$ and $\Sigma \setminus \Sigma^{\tau}$ is disconnected}  \\ 
  g &  \text{ if $a= c = 0$}  \\ 
 \end{cases} $$ 
\end{cor}

\subsection{Fundamental groups and the proof of Theorems \ref{thm6} and \ref{thm3}}

In this section we compute $\pi_1(BC\G_\R) = \pi_0(C\G_\R) $ and $\pi_1(B\G_\R) = \pi_0(\G_\R)$. 

We begin with $\pi_0(S\G_\R) = \pi_0(S\mathcal{U}_\R)$.  By Lemma \ref{gaugegroupconstr} we have a short exact sequence
$$ Maps_*(X/\partial X, SU(r) ) \rightarrow S\mathcal{U}_\R \rightarrow \prod_{i=1}^n LSU(r)^{\tau_i}  $$
Observe that $X/\partial X$ is a 2-dimensional cell complex and $SU(r)$ is 2-connected, so $Maps_*(X/\partial X, SU(r) )$ is path connected. It follows that
$$ \pi_0( S\G_\R ) \cong  \pi_0(S\mathcal{U}_\R ) \cong \prod_{i=1}^n \pi_0( LSU(r)^{\tau_i}) .$$

For real loop groups of type a and b we have a fibration sequence
$$ \Omega SO(r) \rightarrow LSU(r)^{\tau_i} \rightarrow SO(r). $$
Since $SO(r)$ is connected, we see $\pi_0(LSU(r))$ is the cokernel of a homomorphism $\pi_1(SO(r)) \rightarrow \pi_0(\Omega SO(r))= \pi_1(SO(r))$.
For $r\geq 3$ this is the cokernel of a map $ \Z/2 \rightarrow \Z/2$ which must be $\Z/2$ since $H^1(BLSO(r);\F/2) = \Z_2$.
For $r = 2$, we get the cokernel of a map from $\pi_1(SO(2)) \cong \Z$ to itself which the reader can check gives $\Z$ for type a and $\Z/2$ for type b.

For type c we have $\Omega SU(r) \rightarrow LSU(r)^{\tau_c} \rightarrow SU(r)$  which implies $\pi_0(LSU(r)) =1$. Therefore
\begin{prop}\label{pizerosg}
Suppose $(\Sigma,\tau)$ is a real curve with $a$ real circles and $\xi$ is a real line bundle for which $b$ circles are odd.
We have an isomorphism
$$  \pi_0(S\G_\R) = \pi_1(BS\G_\R) \cong    \begin{cases}  (\Z/2)^a & \text{if $r\geq 3$}\\ (\Z/2)^b \times (\Z)^{a-b} & \text{if $r = 2$}.\\ \end{cases} $$
\end{prop}

\begin{prop}\label{pizerocg}
 We have an isomorphism 
 $$    \pi_0(C\G_\R ) \cong \Z/2 \ltimes \pi_0(S\G_\R) $$
 where $\Z/2$ acts on $\pi_0(S\G_\R)$ diagonally: trivially on the $\Z/2$ factors and by $-1$ on the $\Z$ factors.
\end{prop}

\begin{proof}[Proof of \ref{pizerocg}]
Since (\ref{SESCSR}) splits we know $ \pi_0(C\G_\R ) \cong \Z/2 \ltimes \pi_0(S\G_\R) $ is a semi-direct product.  In terms of the isomorphism in Proposition \ref{pizerosg}, $\Z/2$ acts by conjugating each $LU(r)^{\tau_i}$ by a constant real matrix with negative determinant, producing an automorphism of  $\pi_0( LSU_{r}^{\tau_i})$. Clearly the automorphism is trivial whenever $\pi_0( LSU_{r}^{\tau_i}) \cong \Z/2$. The one remaining case is when $ LSU_{r}^{\tau_i} = LSO_2$ in which case $\pi_0(LSO_2) = \pi_0(\Omega SO_2) \cong \Z$ and the involution acts by $-1$. 
\end{proof}

\begin{proof}[Proof of Theorem \ref{thm3}]
According to Corollary \ref{toprovetheorem3}, we have an isomorphism $\pi_1( B\overline{C \G}_{\R} ) \cong \pi_1( M(2,\xi)^{\tau})$ and according to (\ref{noncansplit}) we have $\pi_1( BC \G_{\R} )  \cong  \pi_1( B\overline{C \G}_{\R} ) \times \Z/2$. The abelianization of $\pi_1( B C \G_{\R} ) $ is $(\Z/2)^{a+1}$ so $H_1(M(2,\xi)^{\tau}) \cong (\Z/2)^a$. 
\end{proof}

\begin{proof}[Proof of Theorem \ref{thm6}]
Assume that $H(M(r, \xi_1)^{\tau_1};\Z)\cong H(M(r, \xi_2)^{\tau_2};\Z)$. Since $M(r, \xi_i)^{\tau_i}$ is a closed manifold of dimension $(r^2)(2g_i-2)$ it follows that $g_1=g_2$. The equality $a_1=a_2$ follows from Theorem \ref{thm3}. Denote $g:=g_1=g_2$ and $a:= a_1=a_2$.

Assume further that $r$ is even and either $r=2$ and $g_1 \geq 5$ or $r\geq 4$ and $g_1 \geq 3$. Recall that by assumption $gcd(r,d_i)=1$, so $d_i$  is odd.  By (\ref{oddSW}) $a-c_i$ must also be odd, so in particulary $a> c_i$ and $c_1-c_2$ is even. By Corollary \ref{toprovetheorem3} we have an isomorphism $H^k(M(r, \xi)^{\tau};\F) \cong H^k( B\overline{C\G}_\R; \F )$ for all $k \leq g(r-1)-2 \leq 2r-1$. By the universal coefficient theorem and (\ref{noncansplit}) we have isomorphisms
$$ H^k(M(r, \xi)^{\tau}; \F ) \cong  H^k(  B\overline{C\G}_\R; \F ) \cong  H^k(  BC\G_\R; \F) $$
where $\F$ has characteristic $\neq 2$ and $k \leq 2r-1$. Assume without loss of generality that $c_1 \geq c_2$.  From the formula for $\beta_{2r-2}$ in Corollary \ref{Bettis2r-2}, it follows that either $c_1=c_2$, or $c_1 \in \{1,2\}$ and $c_2=0$. The case $c_1=1$ can be dismissed because $c_1-c_2$ must be even. If $c_1=2$ then $\beta_{2r-1} = g+1$ and if $c_2 =0$ then  $\beta_{2r-1} = g$ or $g-1$ so this case also leads to a contradiction.

Assume further that $c= c_1 =c_2$ is even and $g \geq 6$. By the coprime condition and (\ref{d=w}), the number of odd circles must be odd so $a$ must also be odd. For fixed odd $g$ and odd $a$ there is only one possible topological type of $(\Sigma, \tau)$ so the result holds. For fixed even $g$ and odd $a$ there are two topological types for $(\Sigma,\tau)$ distinguished by whether $\Sigma \setminus \Sigma^{\tau}$ is connected or disconnected. By Theorem \ref{BigLongThm}, if $c= 0$ then the corresponding $B C \G_\R$ have different Betti numbers in degree $2r-2$ while if $c>0$ they have different Betti numbers in degree $4r-4$.  Since $2r-2 \leq 4r-4 \leq g(r-1)-2$, it follows that the moduli spaces have different Betti numbers in degree $4r-4$. 
\end{proof}

\section{Equivariant perfection and the proof of Theorem \ref{thm1}}\label{Equivariant perfection}

The goal of this section is to prove the following theorem.

\begin{thm}\label{eqperf}
The real-Harder Narsimhan stratification,
 \begin{equation}\label{RHNstrat}
 \pC^{\ttau} = \cup_{\mu} C_{\mu}^{\ttau}
 \end{equation} 
 is $C\G_\R$-equivariantly perfect with respect to $\Z_2$-coefficients. Consequently the induced map $H_{C\G_\R}(\pC^{\ttau}; \Z_2) \rightarrow H_{C\G_\R}(\pC^{\ttau}_{ss}; \Z_2)$ is surjective.
\end{thm}

The analogous result with $C\G_\R$ replaced with $\G_\R$ was proven in \cite{B}. That proof boils down to showing that the equivariant Euler classes of the normal bundles of each stratum $C_\mu^{\ttau}$ is not a zero divisor in the cohomology ring $H^*(C_\mu^{\ttau};\Z_2 )$. This was accomplished using the following version of the Atiyah-Bott Lemma (Lemma 3.1 from \cite{B}):

\begin{lem}\label{aoelurcalorce}
Let $G$ be a compact connected Lie group with $H^*(G;\Z)$ torsion free. Let $X$ be a $G$-space of finite type and let $E\rightarrow X$ be a $G$-equivariant $\R^n$-vector bundle. Suppose that there exists $\epsilon \in G$ such that
\begin{itemize}
	\item $\epsilon^2$ is the identity in $G$
	\item $\epsilon$ acts trivially on $X$
	\item $\epsilon$ acts by scalar multiplication by $-1$ on $E$.
\end{itemize} 
Then the equivariant Euler class $Eul_G(E)$ is not a zero divisor in $H^*_G(X;\Z_2)$.
\end{lem}
Unfortunately, the required element $\epsilon \in \G_\R$ does not lie in $C\G_\R$. For that reason we must replace $C \G_\R$ with a larger group containing $\epsilon$. Recall (\ref{defineCG}) that $C \G_\R$ is equal to the kernel of the natural homomorphism $ \G_{\R} \rightarrow \overline{\G(1)}_{\R}$.  Define $\widetilde{C\G}_\R$ by the short exact sequence
$$  1 \rightarrow \widetilde{C\G}_\R \rightarrow \G_\R \rightarrow \pi_0(\overline{\G(1)}_{\R}) \rightarrow 1.$$

\begin{prop}
The  inclusion $C\G_\R \hookrightarrow \widetilde{C \G}_\R$ is a weak homotopy equivalence. Consequently, (\ref{RHNstrat}) is $C\G_\R$-equivariantly perfect if and only if it is  $\widetilde{C \G}_\R$-equivariantly perfect.
\end{prop}

\begin{proof}
It is clear from the definition that the coset space $ \widetilde{C \G}_\R/ C\G_\R $ homeomorphic to the identity component of $\overline{\G(1)}_{\R}$, which was proven to be contractible in Lemma \ref{homtype1}.
\end{proof}

\begin{lem}\label{surjonpione}
For every splitting $ (E,\ttau) = (D_1, \ttau_1) \oplus ... \oplus (D_k,\ttau_k) $ into $C^{\infty}$-Real bundles we have a surjection

$$ \pi_0(\G(D_1)_\R) \times ... \times \pi_0(\G(D_k)_\R) \rightarrow \pi_0(\G(E)_\R). $$

\end{lem}

\begin{proof}

It suffices to show that the resticted map $i: \pi_0(\G(D_1)_\R) \rightarrow \pi_0(\G(E)_\R)$ is surjective. Consider the short exact sequence (\ref{defineCG}).  Since $\overline{\G(1)}_\R$ is a $K(\Z^g, 1)$, we have a diagram
$$\xymatrix{  1 \ar[r] & \pi_0( C\G(D_1)_\R) \ar[d]^{i'} \ar[r]& \pi_0(\G(D_1)_\R) \ar[d]^i \ar[r] & \pi_0(\overline{\G(1)}_\R) \ar[d]^= \ar[r]& 1 \\ 
1 \ar[r] & \pi_0( C\G(E)_\R) \ar[r]& \pi_0(\G(E)_\R) \ar[r]& \pi_0(\overline{\G(1)}_\R) \ar[r]& 1} . $$
The surjectivity of $i'$ is evident from the description in Proposition \ref{pizerocg}. The surjectivity of $i$ follows.

\end{proof}

The following lemma is necessary for the induction step in the calculation of Betti numbers.

\begin{lem}\label{DecompBund}
Let $(D_1, \ttau_1) \oplus ... \oplus (D_k,\ttau_k) = (E,\ttau) = (D_1', \ttau_1') \oplus ... \oplus (D_k',\ttau_k')$ be two different decompositions of $E$ into $C^{\infty}$-Real subbundles, such that $(D_i,\ttau_k)  \cong (D_i',\ttau_k') $ for all $i$. Then there exists $ g \in \widetilde{C\G(E)}_\R$ such that $g(D_i) = D_i'$ for all $i$.
\end{lem}

\begin{proof}
It is clear that simply by summing together this isomorphisms $D_i \cong D_i'$ that we can find a gauge transformation $g \in \G(E)_\R$ satisfying $g(D_i) = D_i'$. The only question is whether we can choose $g \in \widetilde{C\G_\R}$.  But by Lemma \ref{surjonpione}, we can compose $g$ by an element of $h \in \G(D_1')_\R \times ... \times \G(D_k')_\R$ so that $hg$ lies in the identity component of $\G_\R$ hence must also lie in $\widetilde{C\G_\R}$.
\end{proof}

\begin{proof}[Proof of Theorem \ref{eqperf}]

By Lemma \ref{DecompBund}, $C\G_{\R}$ acts transitively on the set of decompositions $(D_1, \ttau_1) \oplus ... \oplus (D_k,\ttau_k)$ of a given topological type. It follows that there is a homotopy equivalence of homotopy quotients 
\begin{equation}\label{stratasplitting}
(C_{\mu}^{\ttau})_{h \widetilde{C\G}_\R^{\ttau}} \cong  (\pC_{ss}^{\tau_1} \times ... \times \pC_{ss}^{\tau_k})_{h\widetilde{C\G}(D_1,...,D_k)_\R}
\end{equation} 
where $$\widetilde{C\G}(D_1,...,D_k)_\R = \G(D_1)_\R \times...\times \G(D_k)_\R \cap \widetilde{C\G}_\R.$$ 

Choose a basepoint $p \in \Sigma$ that is not fixed by $\tau$. Then restricting gauge transformations to the fiber over $p$ determines a short exact sequence
$$1 \rightarrow  \widetilde{C\G}_{bas}(D_1,...,D_k)_\R \rightarrow \widetilde{C\G}(D_1,...,D_k)_\R \rightarrow \prod_{i=1}^k GL_{r_i}(\C) \rightarrow 1 $$ 
By forming the homotopy quotient in stages we get an isomorphism
$$ H^*_{ \widetilde{C\G}_\R^{\ttau}}( C_{\mu}^{\ttau} ;\Z_2  )  \cong H^*_U(  (\pC_{ss}^{\tau_1} \times ... \times \pC_{ss}^{\tau_k})_{h\widetilde{C\G}_{bas} (D_1,...,D_k)_\R} ; \Z_2 ) $$
where $U = U(r_1) \times...\times U(r_k)$ is the maximal compact subgroup of $\prod_{i=1}^k GL_{r_i}(\C)$. As explained in (\cite{B} (2.13)), the normal bundle decomposes into a direct sum of subbundles $N = \bigoplus_{i < j} N_{i,j}$ where the fibre  $(N_{i,j})_{(\bar{\partial}_1,...\bar{\partial}_k)} \cong H^1( D_i^* \otimes D_j, \bar{\partial}_1^* \otimes Id_{D_j} + Id_{D_i^*} \otimes \bar{\partial}_2))^{\tau_i^* \otimes \tau_j}$. 
The element $(Id_{r_1},...,-Id_{r_i},...,Id_{r_k}) \in U$ acts by $-1$ on the summand $N_{i,j}$ and trivially on the base so by Lemma \ref{aoelurcalorce}, the equivariant Euler class $Eul_U(N) = \prod_{i< j} Eul_U(N_{i,j})$ is not a zero divisor in $H^*_{ \widetilde{C\G}_\R^{\ttau}}( C_{\mu}^{\ttau} ;\Z_2  )$.
\end{proof}

\begin{proof}[Proof of Theorem \ref{thm1}]
Combining Theorem \ref{eqperf}  and Theorem \ref{surjf} implies that the composed map $H^*(B \G_\R;\Z_2 ) \rightarrow H^*( (\pC_{ss})_{hC\G_\R^{\ttau}};\Z_2)$ is surjective. From (\ref{noncansplit}) and Corollary \ref{homotequivMC}, it follows that $H^*(B \overline{C\G}_\R;\Z_2 ) \rightarrow H^*( M(r,\xi)^{\tau};\Z_2)$ is also surjective. From (\ref{LHdiagram}) $H^*(M(r,d)^{\tau}_w;;\Z_2 ) \rightarrow H^*( M(r,\xi)^{\tau};\Z_2)$ is also surjective, so the result follows by the Leray-Hirsch Theorem.
\end{proof}

\section{ Proof of Theorem \ref{thm4}}\label{Proof of thm4}

The following is adapted mutatis-mutandis from \cite{B3}.  See \S 4 of that paper for a more detailed proof. The idea is simply that the normal bundles of the unstable strata are non-orientable, which implies that their Thom spaces must be acyclic, so they contribute nothing the the Morse complex. 

\begin{prop}
Let $(E,\ttau)$ be a Real bundle of rank 2 and let $\F$ be a field of odd characteristic. Suppose that $g \equiv d~mod~2$. Then there is an isomorphism
$$ H^*(B C\G_{\R}; \F)  \cong H^*_{C\G_{\R}}( \pC_{ss}^{\ttau};\F). $$
\end{prop}

\begin{proof}
The action of $C \G_\R$ preserves the real Harder-Narasimhan stratification $\pC^{\ttau} = \bigcup_{\mu} \pC_\mu^{\ttau} $, and determines a stratification
$$ BC\G_{\R} \cong (\pC^{\ttau} )_{h C\G_\R} =  \bigcup_{\mu} (\pC_\mu^{\ttau} )_{h C\G_\R} = (\pC_{ss}^{\ttau} )_{h C\G_\R} \cup  \Big(  \bigcup_{\mu \neq ss} (\pC_\mu^{\ttau} )_{h C\G_\R}  \Big)$$
Since $E$ has rank $2$ the higher strata correspond to $C^{\infty}$-decompositions into Real line bundles $E = L_1\oplus L_2$. The corresponding strata have the form ((\ref{stratasplitting}) up to homotopy)
$$ (\pC_\mu^{\ttau} )_{h C \G_{\R}} \cong  (\pC(L_1)_{ss}^{\ttau_1} \times \pC(L_2)^{\ttau_2})_{h C \G(L_1,L_2)_{\R}} .$$
where $C \G(L_1,L_2)_{\R} := ( \G(L_1)_\R \times \G(L_2)_\R) \cap C\G_\R$ . 
For line bundles, $\pC(L_i)^{\ttau_i}_{ss} =\pC(L_i)^{\ttau_i}$ is contractible and we have an isomorphism $\G(1)_{\R} \times \R^* \cong C \G(L_1,L_2)_{\R} $ defined by $ (g, \lambda) \mapsto g \oplus \lambda g^{-1}$ so by Lemma \ref{homtype1} and the fact that $\G(1)_\R \cong  \overline{\G_\R} \times \R^*$ we have
$$  (\pC_\mu^{\ttau} )_{h C \G_{\R}}  \cong  B \G(1)_\R  \cong (S^1)^g \times (\R P^{\infty})^2 = K( Z^{2g} \times (\Z/2)^2, 1).$$
It follows that for every non-trivial 2-fold covering map over $(\pC_\mu^{\ttau})_{h C \G_{\R}} $ induces an isomorphism in $\F$-cohomology or equivalently, that every non-trivial rank one $\F$-local system over $(\pC_\mu^{\ttau} )_{h C \G_{\R}} $ is acyclic. This implies that the Thom space of any non-orientable vector bundle over $ (\pC_\mu^{\ttau} )_{h C \G_{\R}}$ must be $\F$-acyclic. The normal bundle of $\pC_\mu^{\ttau}$ has odd rank and the constant scalar $1 \oplus -1 \in C \G(L_1,L_2)_{\R} $ acts by scalar multplication by $-1$ on $N$, so the restriction of $N_{h C \G_{\R}}$ to the second factor of $\R P^{\infty}$ is non-orientable, thus the Thom space of $N_{h C \G_{\R}}$ is $\F$-acyclic. Since this is true for every stratum except the semi-stable stratum, the result is proven. 
\end{proof}

\begin{proof}[Proof of Theorem \ref{thm4}]
Since $\R P^{\infty}$ is $\F$-acyclic, by Corollary \ref{homotequivMC} we have $$H^*(B C\G_{\R}; \F)  \cong H^*_{C\G_{\R}}( \pC_{ss}^{\ttau};\F) \cong  H^*( M(2,\xi) \times \R P^{\infty} ;\F)\cong H^*( M(2,\xi) ;\F).$$ The Poincar\'e series can be read off from Theorem \ref{BigLongThm}.
\end{proof}

\section{Betti numbers}\label{Betti numbers and fundamental groups}

We are now able to compute some Poincar\'e polynomials. For rank $r=1$, $M(1,\xi) = M(1,\xi)^{\tau}$ is just a point. The first interesting case is when $r=2$.

\begin{prop}
Let $\Sigma$ be a genus $g$ real curve with $a >0$ real path components and let $\xi$ be a Real line bundle over $(\Sigma, \tau)$ of odd degree. The moduli space $M(2,d, \tau)$ of real bundles of rank two, odd degree $d$ and fixed topological type has Poincar\'e series

\begin{equation}\label{ercaboecub}
P_t(M(2,\xi)^{\tau}; \Z_2)= \frac{(1+t)^{a-1}(1+t^2)^{a-1}(1+t^3)^{g-\alpha} - 2^{a-1} t^{g}(1+t)^{g}}{(1-t)(1-t^2)}
\end{equation}
where $a = \pi_0(\Sigma^\tau)$. 
\end{prop}

For example, for a real curve of genus $g=2$, respectively $a= 1,2,3$ real circles, $P_t(M(2,\xi)^\tau;\Z_2)$ equals
$$t^3+t^2+t+ 1 $$
$$t^3+2t^2+2t+1 $$
$$t^3+3t^2+3t+1 $$

For a real curve of genus $g=3$, $a=1,2,3,4$ real circles, $P_t(M(2,\xi)^\tau;\Z_2)$ equals
$$t^6+t^5+2t^4+4t^3+2t^2+t+1 $$
$$t^6 + 2t^5 + 4t^4 + 6t^3 + 4t^2 + 2t + 1 $$
$$t^6 + 3t^5 + 7t^4 + 10t^3 + 7t^2 + 3t + 1$$
$$t^6 + 4t^5 + 11t^4 + 16t^3 + 11t^2 + 4t + 1 $$

This can be compared $P_t(M(2,\xi)^\tau;\Z_p)$, for $p$ odd, for a real curve of genus $g=3$ with respectively $c=0,1,2,3$ even circles:
$$t^6+2t^3+1 $$
$$t^6+2t^3+1 $$
$$t^6+t^4+4t^3+t^2+1 $$
$$t^6+3t^4+8t^3+3t^2+1 $$

For a real curve of genus $g=4$, $a=1,2,3,4,5$ real circles, $P_t(M(2,\xi)^\tau;\Z_2)$ equals
$$t^9 + t^8 + 2t^7 + 6t^6 + 6t^5 + 6t^4 + 6t^3 + 2t^2 + t + 1 $$
$$t^9 + 2t^8 + 4t^7 + 9t^6 + 12t^5 + 12t^4 + 9t^3 + 4t^2 + 2t + 1 $$
$$t^9 + 3t^8 + 7t^7 + 15t^6 + 22t^5 + 22t^4 + 15t^3 + 7t^2 + 3t + 1$$
$$t^9 + 4t^8 + 11t^7 + 25t^6 + 39t^5 + 39t^4 + 25t^3 + 11t^2 + 4t + 1 $$
$$ t^9 + 5t^8 + 16t^7 + 40t^6 + 66t^5 + 66t^4 + 40t^3 + 16t^2 + 5t + 1 $$ 

\begin{prop}
Let $\Sigma$ be a genus $g$ real curve with $a >0$ real path components and let $\xi$ be a Real line bundle over $(\Sigma, \tau)$ of degree not divisible by three. The moduli space $M(2,\xi)^\tau$ has $\Z_2$ Poincar\'e series

\begin{eqnarray*} P_t(M(3,\xi)^{\tau};\Z_2) =  \frac{(1+t)^{b} (1+t^2)^{2b}(1+t^3)^{g}(1+t^5)^{g-b}}{(1-t)(1-t^2)^2(1-t^3)}\\    
- 2^b\frac{t^{2g}(1+t)^{g+b}(1+t^2)^b(1+t^3)^{g-b} }{t(1-t)^3(1-t^3)}\\
+4^b\frac{t^{3g}(1+t)^{2g} (1+t^2+t^4)}{ t(1-t)^2(1-t^2)(1-t^6)} \end{eqnarray*}
\end{prop}

In genus $g=2$ and $a=1,2,3$, $P_t(M(3,1,\tau))$ equals

$$ t^8 + t^7 + 3t^6 + 5t^5 + 4t^4 + 5t^3 + 3t^2 + t + 1 $$
$$	t^8 + 2t^7 + 6t^6 + 11t^5 + 12t^4 + 11t^3 + 6t^2 + 2t + 1 $$ 
$$	t^8 + 3t^7 + 10t^6 + 21t^5 + 26t^4 + 21t^3 + 10t^2 + 3t + 1  $$

A closed form formula for the mod 2 Poincar\'e series can be derived from the formula produced by Liu and Schaffhauser \cite{LS} section 6.2.

\section{Orientability and Monotonicity}\label{Orientability and Monotonicity}

\begin{prop}\label{linesreal}
Let $(\Sigma,\tau)$ be a real curve with Real line bundle $\xi$.  Then every element of the Picard group $Pic(M(r,\xi))\cong \Z$ can be represented by a Cartier divisor D such that $\tau(D) = D$.
\end{prop}

\begin{proof}
It was proven by Drezet and Narasimhan (\cite{DN}, Theorem B) that $Pic(M(r,\xi)) \cong \Z$ and is generated by the \emph{theta divisor} $\Theta$ constructed as follows.  Choose a fixed semistable, algebraic vector bundle $\mathcal{F}$ over $\Sigma$ (of some particular rank and degree which is unimportant for our purposes) and define
$$ \Theta:= \{ \E \in M(r,\xi) ~|~H^0(\Sigma, \E \otimes \mathcal{F}) \neq 0\}. $$
Since $\Theta$ is independent of $\mathcal{F}$, so  $\tau(\Theta) = \Theta$. Since every other element of $Pic(M(r,\xi))$ may be represented by $k \Theta$ for some $k \in \Z$, the result follows. 
\end{proof}

\begin{cor}
Let $(\Sigma, \tau)$ be a real curve with Real line bundle $\xi$. The dualizing sheaf $\omega$ of $M(r,\xi)$ is a Real line bundle with fixed point set $\omega^{\tau}$ a topologically trivial $\R^1$-bundle over $M(r,\xi)^{\tau}$.  In particular, if $M(r,\xi)^{\tau}$ is non-singular, then $M(r,\xi)^{\tau}$ is an orientable manifold.
\end{cor}

\begin{proof}
Drezet and Narasimhan (\cite{DN}, Theorem F) prove that the dualizing sheaf $\omega$ on $M(r,\xi)$ is equal to $\omega = \mathcal{O}(-2n \Theta)=  \SL^{\otimes 2}$, where $\SL =  \mathcal{O}(2n \Theta)$. By Lemma \ref{linesreal} $\SL$ is real, so $\omega^{\tau} \cong  (\SL^{\tau})^{\otimes 2}$ is trivial.
\end{proof}

Assume now that $gcd(r,d) =1$. With the standard symplectic structure $(M(r,\xi), \Omega)$ is monotone with $c_1(TM(r,\xi)) = 2 [\Omega]$ in $H^2(M(r,\xi),\Z) \cong \Z$. We have the following easy consequence. 

\begin{prop}
If $gcd(r,d) =1$ then $M(r,\xi)^{\tau}$ is a monotone Lagrangian submanifold of $(M(r,\xi),\Omega)$ with minimal Maslov number a positive multiple of 2. 
\end{prop}

\begin{proof}
The Maslov index of a disk $D \rightarrow M$ bounding a Real Lagrangian is given simply given by $c_1(TM|_{D \cup \tau D})$ (\cite{MS} Theorem C.3.6). In our case, since $c_1(TM) =2 [\Omega]$ is twice another integral class, the minimal Maslov number for disks must be a multiple of 2.  
\end{proof}

\section*{Acknowledgements} Thanks to Indranil Biswas, Shengda Hu for stimulating discussions. This research was funded by an NSERC Discovery grant and was conducted in part at the Tata Institute for Fundamantal Research.

\bibliographystyle{amsalpha}
\bibliography{TomReferences}

\end{document}